\documentclass[11pt,twoside]{article}
\usepackage[margin=1.3in]{geometry}

\usepackage{amsfonts,amsmath,amssymb,amsthm}
\usepackage{CJK,CJKulem,CJKnumb,cases,cite,color,curves}
\usepackage{bm,dsfont,enumitem,extarrows,fontenc,graphicx,ifthen,latexsym,listings,makeidx,mathrsfs,psfrag,times,xcolor}
\usepackage[colorlinks,citecolor=red,linkcolor=red]{hyperref} 
\usepackage[title]{appendix}

\let\oldbibliography\thebibliography
\renewcommand{\thebibliography}[1]{\oldbibliography{#1}\setlength{\itemsep}{0pt}}

\numberwithin{equation}{section}

\makeindex
\newtheorem{theorem}{Theorem}[section]

\newtheorem{lemma}[theorem]{Lemma}

\newtheorem{remark}[theorem]{Remark}
\newtheorem{definition}[theorem]{Definition}

\newcommand{\R}{\mathbb R}

\newcommand{\Sp}{\mathbb S}

\begin{document}
\title{\bf A priori estimates versus arbitrarily large solutions for fractional semi-linear elliptic equations with critical Sobolev exponent}
\author{Xusheng Du,  \;  Hui Yang}
\maketitle

\begin{center}
\begin{minipage}{130mm}
\begin{center}{\bf Abstract}\end{center}
We study positive solutions to the fractional semi-linear elliptic equation
$$
(- \Delta)^\sigma u = K(x) u^\frac{n + 2 \sigma}{n - 2 \sigma} ~~~~~~ \textmd{in} ~ B_2 \setminus \{ 0 \} 
$$
with an isolated singularity at the origin, where $K$ is a positive function on $B_2$, the punctured ball $B_2 \setminus \{ 0 \} \subset \R^n$ with $n \geq 2$, $\sigma \in (0, 1)$, and $(- \Delta)^\sigma$ is the fractional Laplacian. In lower dimensions, we show that, for any $K \in C^1 (B_2)$, a positive solution $u$ always satisfies that $u(x) \leq C |x|^{ - (n - 2 \sigma)/2 }$ near the origin. In contrast, we construct positive functions $K \in C^1 (B_2)$ in higher dimensions such that a positive solution $u$ could be arbitrarily large near the origin. In particular, these results also apply to the prescribed boundary mean curvature equations on $\mathbb{B}^{n+1}$.

\vskip0.10in


\vskip0.10in

\end{minipage}
\end{center}

\vskip0.20in

\section{Introduction}

In this paper, we are interested in the singular positive solutions to the fractional semi-linear elliptic equation
\begin{equation}\label{ukb2}
(- \Delta)^\sigma u = K(x) u^\frac{n + 2 \sigma}{n - 2 \sigma} ~~ \textmd{in} ~ B_2 \setminus \{ 0 \}, ~~~~ u > 0 ~~ \textmd{in} ~ \R^n \setminus \{ 0 \},
\end{equation}
where $K$ is a positive continuous function on $B_2$, the punctured ball $B_2 \setminus \{ 0 \} \subset \R^n$ with $n \geq 2$, $\sigma \in (0, 1)$, and $(- \Delta)^\sigma$ is the fractional Laplacian defined as
\begin{equation}\label{FL}
(- \Delta)^\sigma u(x) = C_{n, \sigma} {\rm P.V.} \int_{\R^n} \frac{u(x) - u(y)}{ |x - y|^{n + 2 \sigma} } dy
\end{equation}
with $C_{n, \sigma} = \frac{ 2^{2 \sigma} \sigma \Gamma \left( \frac{n + 2 \sigma}{2} \right) }{ \pi^\frac{n}{2} \Gamma (1 - \sigma) }$ and the gamma function $\Gamma$. This equation with the critical Sobolev exponent arises in the study of the fractional Yamabe problem \cite{CG-11,GMS,GQ} and fractional Nirenberg problem \cite{JLX-14, JLX-15}. More precisely, every solution $u$ of \eqref{ukb2} induces a metric $g : = u^{4/(n - 2\sigma)} |dx|^2$ conformal to the flat metric on $\R^n$ such that $K(x)$ is the fractional $Q$-curvature \cite{CG-11} of the new metric $g$. An interesting question is the following: If a solution $u(x)$ of \eqref{ukb2} has a non-removable singularity at $\{ 0 \}$, how does it tend to infinity as $x$ approaches the origin?

This question in the Laplacian case $\sigma = 1$ was initially studied by Caffarelli, Gidas and Spruck in \cite{CGS} when $K$ is identically a positive constant. They proved that every positive solution $u$ is asymptotically radially symmetric and has a precise asymptotic behavior near the isolated singularity $0$. In particular, their result implies that $u$ satisfies the following local estimate near $0$ 
\begin{equation}\label{R=CaGs}
u(x) \leq C |x|^{ - \frac{n - 2}{2} }. 
\end{equation}
We may also see the work of Korevaar-Mazzeo-Pacard-Schoen \cite{KMPS} for a different proof in this classical case. When $K$ is a non-constant positive function, Chen and Lin \cite{CL-97,Lin-00} established \eqref{R=CaGs} for positive solutions to \eqref{ukb2} in the case $\sigma = 1$ under certain flatness conditions at critical points of $K$ by using the method of moving planes. Later, Taliaferro and Zhang \cite{TZ-06,Zhang-02} further explored the flatness conditions on $K$ such that any positive solution of \eqref{ukb2} with $\sigma = 1$ satisfies the local estimate \eqref{R=CaGs} via the moving sphere method.

When $\sigma \in (0, 1)$ and $K$ is identically a positive constant, Caffarelli, Jin, Sire and Xiong \cite{CJSX} proved the following local estimate 
\begin{equation}\label{R=CaGs098}
u(x) \leq C |x|^{ - \frac{n - 2 \sigma}{2} }
\end{equation}
for positive solutions of \eqref{ukb2} near the singularity $0$ relying on the extension formulations of fractional Laplacians established by Caffarelli and Silvestre \cite{CS-07}.   Based on this estimate, they also showed that every solution $u$ of \eqref{ukb2} is asymptotically radially symmetric. Further, it is natural to consider the case where $K$ is a non-constant function and ask that under what assumptions on $K$ every singular solution $u$ of \eqref{ukb2} satisfies the estimate \eqref{R=CaGs098} near the origin. In a recent paper \cite{JY}, Jin and Yang established local estimates for higher order conformal $Q$-curvature equation by studying corresponding integral equation which, in particular, in the case $\sigma \in (0, 1)$ is closely related to the fractional equation \eqref{ukb2}. However, as pointed out in Remark 1.8 there, the integral equation in the case $\sigma \in (0, 1)$ encounters a difficulty due to the more singular properties of an integral kernel, and it was not covered in \cite{JY}. The first goal of this paper is to derive the local estimate \eqref{R=CaGs098} for the fractional equation \eqref{ukb2} in lower dimensions.

We study \eqref{ukb2} via the well-known extension formulations for fractional Laplacians in \cite{CS-07}, through which we can consider a degenerate but local elliptic equation with a Neumann boundary condition in one dimension higher. To be more precisely, we first introduce some notations. We use capital letters, such as $X = (x, t) \in \R^n \times \R$,  to denote points in $\R^{n + 1}$. We denote $\mathcal{B}_R $ as the open ball in $\R^{n + 1}$ with radius $R$ and center $0$, $\mathcal{B}_R^+$ as the upper half ball $\mathcal{B}_R \cap \R_+^{n + 1}$, and $B_R $ as the open ball in $\R^n$ with radius $R$ and center $0$. We also denote $\partial' \mathcal{B}_R^+ $ as the flat part of $\partial \mathcal{B}_R^+$ which is the ball $B_R$ in $\R^n$. Then instead of \eqref{ukb2} we study 
\begin{equation}\label{UKB2}
\left\{
\aligned
{\rm div} (t^{1 - 2 \sigma} \nabla U) & = 0 & \textmd{in} & ~ \mathcal{B}_2^+, \\
\frac{\partial U}{\partial \nu^\sigma} (x, 0) & = K(x) U(x, 0)^\frac{n + 2 \sigma}{n - 2 \sigma} ~~~~~~ & \textmd{on} & ~ \partial' \mathcal{B}_2^+ \setminus \{ 0 \}, 
\endaligned
\right.
\end{equation}
where $\frac{\partial U}{\partial \nu^\sigma} (x, 0) = - \lim\limits_{t \to 0^+} t^{1 - 2 \sigma} \partial_t U(x, t)$. By \cite{CS-07} we only need to derive a local estimate for the trace 
$$
u(x) : = U(x, 0)
$$
of a non-negative solution $U(x, t)$ of \eqref{UKB2} near the origin.

When $\sigma = 1/2$, the equation \eqref{UKB2} appears in the study of prescribing mean curvature on $\partial \mathbb{B}^{n + 1}$ and zero scalar curvature in $\mathbb{B}^{n + 1}$; see, for example, \cite{CXY, DMA, Es-96, EG}. In this case, the equation is without weight and thus elliptic.

We say that $U$ is a weak solution of \eqref{UKB2} if $U$ is in the weighted Sobolev space $W^{1, 2} (t^{1 - 2 \sigma}, \mathcal{B}_2^+ \setminus \overline{\mathcal{B}_\varepsilon^+})$ for any $0 < \varepsilon < 2$ and it satisfies
\begin{equation}\label{nayhdy63k}
\int_{\mathcal{B}_2^+} t^{1 - 2 \sigma} \nabla U \nabla \Psi = \int_{\partial' \mathcal{B}_2^+} K U^\frac{n + 2 \sigma}{n - 2 \sigma} \Psi
\end{equation}
for every $\Psi \in C_c^\infty ( (\mathcal{B}_2^+ \cup \partial' \mathcal{B}_2^+) \setminus \{ 0 \} )$.

Before stating our first theorem, we introduce a notation $\mathcal{C}^\alpha (B_2)$ with $\alpha \in (0, 1]$.

\begin{definition} For $\alpha \in (0, 1)$, $\mathcal{C}^\alpha (B_2)$ is the set of all functions $f \in C(B_2)$ satisfying
$$
|f(x) - f(y)| \leq c(|x - y|) |x - y|^\alpha ~~~~~~ \textmd{for all} ~ x, y \in B_2,
$$
where $c(\cdot)$ is a non-negative continuous function with $c(0) = 0$. For $\alpha = 1$, $\mathcal{C}^1 (B_2)$ is the usual space $C^1(B_2)$. 
\end{definition}

Our first result is the following local estimate for non-negative solutions of \eqref{UKB2} in dimension $n = 2$ or $3$.

\begin{theorem}\label{Pri} 
Suppose that $\sigma \in [1/2, 1)$, $n = 2$ or $3$, and $K \in \mathcal{C}^\alpha (B_2)$ is a positive function with $\alpha = (n - 2 \sigma)/2$.  If $\sigma = 1/2$ and $n = 3$, then we additionally suppose that $\nabla K(0) = 0$.  Let $U$ be a positive weak solution of \eqref{UKB2}. Then there exists a constant $C > 0$ such that 
\begin{equation}\label{vnth89mk}
u(x) \leq C |x|^{ - \frac{n - 2 \sigma}{2} } 
\end{equation} 
for all $x \in B_1 \setminus \{ 0 \}$. 
\end{theorem}

\begin{remark}
In particular, for the case $\sigma = 1/2$, the local estimate \eqref{vnth89mk} holds for any positive solution of the boundary mean curvature equation \eqref{UKB2} when $K \in \mathcal{C}^{1/2} (B_2)$ in dimension $n = 2$ or $K \in C^1 (B_2)$ with $\nabla K(0) = 0$ in dimension $n = 3$.
\end{remark}

We will prove Theorem \ref{Pri} using the method of moving spheres introduced by Li-Zhu \cite{LZ-95}. See also \cite{CJSX, JLX-14, JDSX, JY, LZ-03, TZ-06, Zhang-02} for more applications of the moving sphere method.

On the other hand, supposing only that $K \in C^1 (B_2)$ is a positive function satisfying $\nabla K(0) = 0$, one wonders if the estimate \eqref{vnth89mk} holds in dimension $n \geq 4$. Such problem has been explored by Leung \cite{Leung-03}  and Taliaferro \cite{Ta-05} in the Laplacian case $\sigma = 1$. More precisely, Taliaferro \cite{Ta-05} showed the existence of positive functions $K \in C^1 (B_2)$ in dimension $n \geq 6$ such that \eqref{ukb2} in the case $\sigma = 1$ has a singular solution $u$ which can be arbitrarily large near the origin. Leung \cite{Leung-03} proved the existence of a positive Lipschitz continuous function $K$ on $B_2$ in dimension $n \geq 5$ such that a solution $u$ of \eqref{ukb2} for $\sigma = 1$ does not satisfy $u(x) = O( |x|^{ - (n - 2)/2 } )$ near the origin.

The second goal of this paper is to generalize the result of Taliaferro \cite{Ta-05} to the fractional equation \eqref{ukb2} which, in particular, implies that \eqref{vnth89mk} does not hold when $n > 2 \sigma + 3$ if we only assume that $K \in C^1 (B_2)$ and $\nabla K(0) = 0$.

Now we study the existence of large singular solutions to the extension equation 
\begin{equation}\label{UKRN}
\left\{
\aligned
{\rm div} (t^{1 - 2 \sigma} \nabla U) & = 0 & \textmd{in} & ~ \R_+^{n + 1}, \\
\frac{\partial U}{\partial \nu^\sigma} (x, 0) & = K(x) U(x, 0)^\frac{n + 2 \sigma}{n - 2 \sigma} ~~~~~~ & \textmd{on} & ~ \partial \R_+^{n + 1} \setminus \{ 0 \}.
\endaligned
\right.
\end{equation}

\begin{theorem}\label{Lar} 
Suppose that $\sigma \in (0, 1)$ and $n > 2 \sigma + 3$ is an integer. Let $k : \R^n \to \R$ be a $C^1$ function which is bounded between two positive constants and satisfies $\nabla k(0) = 0$. Let $\varepsilon$ be a positive number and $\varphi : (0, 1) \to (0, \infty)$ be a continuous function. Then there exists a $C^1$ positive function $K : \R^n \to \R$ satisfying $\nabla K(0) = 0$, $K(0) = k(0)$, $K(x) = k(x)$ for $|x| \geq \varepsilon$ and
\begin{equation}\label{K-k}
\| K - k \|_{C^1 (\R^n)} < \varepsilon
\end{equation}
such that \eqref{UKRN} has a positive solution $U$ whose trace satisfies
\begin{equation}\label{u-phi}
u(x) \neq O ( \varphi (|x|) ) ~~~~~~ \textmd{as} ~ |x| \to 0^+
\end{equation}
and
\begin{equation}\label{u-inf}
u(x) = O (|x|^{2 \sigma - n}) ~~~~~~ \textmd{as} ~ |x| \to \infty.
\end{equation}
\end{theorem}

\begin{remark} 
When $\sigma = 1/2$, Theorem \ref{Lar} indicates that there exist positive functions $K \in C^1 (\R^n)$ in dimension $n \geq 5$ such that  the boundary mean curvature equation \eqref{UKRN} has a positive solution which could be arbitrarily large near the singularity $0$. 
\end{remark}

In Theorem \ref{Lar}, the function $\varphi : (0, 1) \to (0, \infty)$ is arbitrarily given and thus its values can be taken to be very large near $0$. Hence, the conclusion of Theorem \ref{Lar} that a solution $U$ of \eqref{UKRN} can be required to satisfy \eqref{u-phi} is in contrast to the result of Theorem \ref{Pri}. Our basic strategy to prove Theorem \ref{Lar} is similar to that introduced by Taliaferro \cite{Ta-05} for $\sigma = 1$, but we first have to set up a framework to fit the nonlocal equation $(- \Delta)^\sigma u = K(x) u^\frac{n + 2 \sigma}{n - 2 \sigma}$ in $\R^n \setminus \{ 0 \}$, and then we extend the constructed solution for this nonlocal equation to \eqref{UKRN} in one dimension higher.

This paper is organized as follows. In Section \ref{Sect02}, we give some basic results for the standard bubble solutions and a Green's formula on the exterior of a half ball with Neumann boundary condition. Section \ref{Sect3} and Section \ref{Sect4} are devoted to the proofs of Theorem \ref{Pri} and Theorem \ref{Lar}, respectively.

\vskip0.10in

\noindent{\bf Acknowledgements.} Both authors would like to thank Professor Tianling Jin for many helpful discussions and encouragement.

\section{Preliminaries}\label{Sect02}

In this section, we introduce some notations and some basic results which will be used in the proof of Theorem \ref{Pri} in the next section.  We denote $\mathcal{B}_R (X)$ as the open ball in $\R^{n + 1}$ with radius $R$ and center $X$, $\mathcal{B}_R^+ (X)$ as $\mathcal{B}_R (X) \cap \R_+^{n + 1}$, and $B_R (x)$ as the open ball in $\R^n$ with radius $R$ and center $x$.  For simplicity we also write $\mathcal{B}_R (0)$, $\mathcal{B}_R^+ (0)$ and $B_R (0)$ as $\mathcal{B}_R$, $\mathcal{B}_R^+$ and $B_R$,  respectively.  For a set $\Omega \subset \overline{\R_+^{n + 1}}$ with boundary $\partial \Omega$, we denote $\partial' \Omega$ as the interior of $\overline{\Omega} \cap \partial \R_+^{n + 1}$ in $\R^n = \partial \R_+^{n + 1}$ and $\partial'' \Omega = \partial \Omega \setminus \partial' \Omega$. Thus, $\partial' \mathcal{B}_R^+ = B_R$ and $\partial'' \mathcal{B}_R^+ = \partial \mathcal{B}_R \cap \overline{\R_+^{n + 1}}$.

In this section, we always assume that $\sigma \in [1/2, 1)$ and $n \geq 2$ is an integer. For $\lambda > 0$, $Y \in \overline{\R_+^{n + 1}} \setminus \{ 0 \}$ and $y \in \R^n \setminus \{ 0 \}$, we denote 
$$
Y^\lambda = \frac{\lambda^2 Y}{|Y|^2} ~~~~~~ \textmd{and} ~~~~~~ y^\lambda = \frac{\lambda^2 y}{|y|^2}.
$$
Let $U : \overline{\R_+^{n + 1}} \to (0, \infty)$ and $u : \R^n \to (0, \infty)$ be two functions, then their Kelvin transformations are defined by 
$$
U^\lambda (Y) = \left( \frac{\lambda}{|Y|} \right)^{n - 2 \sigma} U(Y^\lambda) ~~~~~~ \textmd{and} ~~~~~~ u^\lambda (y) = \left( \frac{\lambda}{|y|} \right)^{n - 2 \sigma} u(y^\lambda).
$$
Let
$$
\widetilde{w} (y) = \left( \frac{1}{1 + |y|^2} \right)^\frac{n - 2 \sigma}{2}
$$
and
$$
\aligned
\widetilde{W} (y, t)  = \mathcal{P}_\sigma [\widetilde{w}] (y, t) & = \int_{\R^n} \mathcal{P}_\sigma (y -z, t) \widetilde{w}(z) dz \\
& = \gamma_{n, \sigma} \int_{\R^n} \left( \frac{1}{1 + |z|^2} \right)^\frac{n + 2 \sigma}{2} \left( \frac{1}{1 + |y - t z|^2} \right)^\frac{n - 2 \sigma}{2} dz,
\endaligned
$$
where 
\begin{equation}\label{Poss09}
\mathcal{P}_\sigma(y, t) =\gamma_{n, \sigma}  \frac{ t^{2 \sigma} }{ (|y |^2 + t^2)^{(n + 2 \sigma)/2} }
\end{equation} 
with a constant $\gamma_{n, \sigma}$ such that  $\gamma_{n, \sigma} \int_{\R^n} (1 + |z|^2)^{- (n + 2 \sigma)/2} dz = 1$.  Then it is well-known that $\widetilde{W}$ solves
$$
\left\{
\aligned
{\rm div} (t^{1 - 2 \sigma} \nabla \widetilde{W}) & = 0 & \textmd{in} & ~ \R_+^{n + 1}, \\
\frac{\partial \widetilde{W}}{\partial \nu^\sigma} (x, 0) & = \widetilde{C}_{n, \sigma} \widetilde{W} (x, 0)^\frac{n + 2 \sigma}{n - 2 \sigma} ~~~~~~ & \textmd{on} & ~ \partial \R_+^{n + 1},
\endaligned
\right.
$$
where $\widetilde{C}_{n, \sigma}$ is a positive constant given as 
\begin{equation}\label{Cns}
\widetilde{C}_{n, \sigma} = \frac{ 2 \Gamma (1 - \sigma) \Gamma \left( \frac{n}{2} + \sigma \right) }{ \Gamma (\sigma) \Gamma \left( \frac{n}{2} - \sigma \right) }.
\end{equation}
Furthermore, we claim that

\begin{lemma}\label{B-BL} Let $\lambda_0 = 1/2$ and $\lambda_1 = 2$. Then there exists $C > 0$ depending only on $n$ and $\sigma$ such that
\begin{equation}\label{B-BL1}
\widetilde{W} (Y) - \widetilde{W}^{\lambda_0} (Y) \geq C (|Y| - \lambda_0) |Y|^{2 \sigma - n - 1} ~~~~~~ \textmd{for} ~ Y \in \overline{\R_+^{n + 1}} \setminus \mathcal{B}_{\lambda_0}^+
\end{equation}
and
\begin{equation}\label{B-BL2}
\frac{ \partial (\widetilde{W} - \widetilde{W}^{\lambda_0}) }{\partial \nu} > C > 0 ~~~~~~ \textmd{on} ~ \partial'' \mathcal{B}_{\lambda_0}^+,
\end{equation}
where $\nu$ denotes the unit outer normal vector of $\partial'' \mathcal{B}_{\lambda_0}^+$. Moreover, we have
\begin{equation}\label{B-BL3}
\widetilde{W} (Y) - \widetilde{W}^{\lambda_1} (Y) < 0 ~~~~~~ \textmd{for} ~ Y \in \overline{\R_+^{n + 1}} \setminus \overline{\mathcal{B}_{\lambda_1}^+}.
\end{equation}
\end{lemma}

\begin{proof} By direct computation, we obtain
$$
\widetilde{w}^{\lambda_0} (y) = \left( \frac{\lambda_0^2}{\lambda_0^4 + |y|^2} \right)^\frac{n - 2 \sigma}{2}.
$$
Then $\widetilde{w} (y) > \widetilde{w}^{\lambda_0} (y)$ for all $y \in \R^n \setminus \overline{B_{\lambda_0}}$. It follows that for $Y \in \R_+^{n + 1} \setminus \overline{\mathcal{B}_{\lambda_0}^+}$,
$$
\aligned
& \widetilde{W} (Y) - \widetilde{W}^{\lambda_0} (Y) \\
= & ~ \gamma_{n, \sigma} \int_{\R^n \setminus B_{\lambda_0}} \bigg( \frac{ t^{2 \sigma} }{ |Y - x|^{n + 2 \sigma} } - \bigg( \frac{\lambda_0}{|x|} \bigg)^{n + 2 \sigma} \frac{ t^{2 \sigma} }{ |Y - x^{\lambda_0}|^{n + 2 \sigma} } \bigg) ( \widetilde{w} (x) - \widetilde{w}^{\lambda_0} (x) ) dx > 0.
\endaligned
$$
We also have for $Y \in \partial'' \mathcal{B}_{\lambda_0}^+ \cap \R_+^{n + 1}$,
$$
\frac{ \partial (\widetilde{W} - \widetilde{W}^{\lambda_0}) }{\partial \nu} (Y) = \gamma_{n, \sigma} (n + 2 \sigma) \int_{\R^n \setminus B_{\lambda_0}} \frac{ t^{2 \sigma} (|x|^2 - \lambda_0^2) }{ \lambda_0 |Y - x|^{n + 2 \sigma + 2} } ( \widetilde{w} (x) - \widetilde{w}^{\lambda_0} (x) ) dx > 0,
$$
and for all $y \in \partial B_{\lambda_0}$,
$$
\frac{ \partial (\widetilde{w} - \widetilde{w}^{\lambda_0}) }{\partial \nu} (y) = (n - 2 \sigma) \left( \frac{1}{1 + \lambda_0^2} \right)^\frac{n - 2 \sigma + 2}{2} \frac{1 - \lambda_0^2}{\lambda_0} > 0.
$$
Note that $\widetilde{W} - \widetilde{W}^{\lambda_0} \in C_{loc}^1 (\overline{\R_+^{n + 1}})$, we obtain that \eqref{B-BL2} holds.  

On the other hand, since  $\widetilde{W}$ is conformally invariant, i.e., $\widetilde{W} (Y) = |Y|^{2 \sigma - n} \widetilde{W} (Y/|Y|^2)$,  we have 
$$
\lim_{|Y| \to \infty} |Y|^{n - 2 \sigma} (\widetilde{W} - \widetilde{W}^{\lambda_0}) (Y) = \widetilde{W} (0) - \lambda_0^{n - 2 \sigma} \widetilde{W} (0) = 1 - \lambda_0^{n - 2 \sigma} > 0.
$$ 
Hence, there exists $C > 0$ depending only on $n$ and $\sigma$ such that \eqref{B-BL1} holds. Using the similar argument, we can prove that \eqref{B-BL3} holds.
\end{proof}

We define
\begin{equation}\label{Green}
G_\lambda (Y, \eta) = N_{n, \sigma} \bigg( |Y - \eta|^{2 \sigma - n} - \bigg( \frac{\lambda}{|\eta|} \bigg)^{n - 2 \sigma} |Y - \eta^\lambda|^{2 \sigma - n} \bigg)
\end{equation}
for $Y = (y, t) \in \R_+^{n + 1} \setminus \mathcal{B}_\lambda^+$ and $\eta \in \R^n \setminus B_\lambda$, where $N_{n, \sigma}$ satisfies
$$
N_{n, \sigma} (n - 2 \sigma) \int_{\R^n} (1 + |z|^2)^{(2 \sigma - n - 2)/2} dz = 1.
$$
Then it is elementary to check that

\begin{lemma}\label{G1} The function $G_\lambda$ satisfies the following:
\begin{enumerate}[label = (\roman*)]
\item $G_\lambda (Y, \eta) > 0$ when $|Y| > \lambda$ and $|\eta| > \lambda$.

\item $G_\lambda (Y, \eta) = 0$ when $|Y| = \lambda$ or $|\eta| = \lambda$.

\item ${\rm div} (t^{1 - 2 \sigma} \nabla_Y G_\lambda) = 0$ for $Y = (y, t) \in \R_+^{n + 1} \setminus \overline{\mathcal{B}_\lambda^+}$ and $\eta \in \R^n \setminus \overline{B_\lambda}$.
\end{enumerate}
\end{lemma}

\begin{lemma}\label{G2} Suppose that $E$ is a smooth bounded domain of $\R^n$ with $B_{2 \lambda} \subset E$ and  $q_\lambda \in C(\overline{E} \setminus B_\lambda)$. Let
$$
\Phi_\lambda (Y) = \int_{E \setminus B_\lambda} G_\lambda (Y, \eta) q_\lambda (\eta) d \eta ~~~~~~ \textmd{for} ~ Y \in \R_+^{n + 1} \setminus \mathcal{B}_\lambda^+.
$$
Then the function $\Phi_\lambda$ satisfies:
\begin{enumerate}[label = (\roman*)]
\item $\Phi_\lambda (Y) = 0$ when $|Y| = \lambda$.

\item $\Phi_\lambda$ solves the equation
$$
\left\{
\aligned
{\rm div} (t^{1 - 2 \sigma} \nabla \Phi_\lambda) & = 0 & \textmd{in} & ~ \R_+^{n + 1} \setminus \overline{\mathcal{B}_\lambda^+}, \\
\frac{\partial \Phi_\lambda}{\partial \nu^\sigma} & = q_\lambda ~~~~~~ & \textmd{on} & ~ E \setminus \overline{B_\lambda}.
\endaligned
\right.
$$
\end{enumerate}
\end{lemma}

\begin{proof} Part (i) follows from Lemma \ref{G1} (ii). The first identity of (ii) follows from Lemma \ref{G1} and Lebesgue's dominated convergence theorem. Now we prove the second identity of (ii). For any $y \in E \setminus \overline{B_\lambda}$, there exists $\delta > 0$ such that $B_\delta (y) \subset E \setminus \overline{B_\lambda}$. Then, for any $0 < r < \delta$,
$$
\aligned
& - t^{1 - 2 \sigma} \partial_t \Phi_\lambda (Y) \\
= & ~ (n - 2 \sigma) N_{n, \sigma} \int_{B_r (y)} t^{2 - 2 \sigma} |Y - \eta|^{2 \sigma - n - 2} q_\lambda (\eta) d \eta \\
& + \bigg[ (n - 2 \sigma) N_{n, \sigma} \int_{E \setminus ( B_\lambda \cup B_r (y) )} t^{2 - 2 \sigma} |Y - \eta|^{2 \sigma - n - 2} q_\lambda (\eta) d \eta \\
& - (n - 2 \sigma) N_{n, \sigma} \int_{E \setminus B_\lambda} t^{2 - 2 \sigma} \bigg( \frac{\lambda}{|\eta|} \bigg)^{n - 2 \sigma} |Y - \eta^\lambda|^{2 \sigma - n - 2} q_\lambda (\eta) d \eta \bigg] \\
= : & ~ I_1 + I_2.
\endaligned
$$
Here we have
$$
\aligned
I_1 = & ~ (n - 2 \sigma) N_{n, \sigma} \int_{B_r (y)} t^{2 - 2 \sigma} |Y - \eta|^{2 \sigma - n - 2} ( q_\lambda (\eta) - q_\lambda (y) ) d \eta \\ 
& + (n - 2 \sigma) N_{n, \sigma} \int_{B_r (y)} t^{2 - 2 \sigma} |Y - \eta|^{2 \sigma - n - 2} q_\lambda (y) d \eta \\
= & ~ (n - 2 \sigma) N_{n, \sigma} \int_{B_r (y)} t^{2 - 2 \sigma} |Y - \eta|^{2 \sigma - n - 2} ( q_\lambda (\eta) - q_\lambda (y) ) d \eta \\
& + (n - 2 \sigma) N_{n, \sigma} q_\lambda (y) \int_{ B_{r/t} } (1 + |z|^2)^{(2 \sigma - n - 2)/2} dz \\
= & ~ I_{11} + I_{12},
\endaligned
$$
where $|I_{11}| \leq \| q_\lambda - q_\lambda (y) \|_{L^\infty ( B_r (y) )}$ and $\lim_{t \to 0^+} I_{12} = q_\lambda (y)$. We also have $\lim_{t \to 0^+} |I_{2}| = 0$. Consequently, we obtain
$$
\limsup\limits_{t \to 0^+} | - t^{1 - 2 \sigma} \partial_t \Phi_\lambda (Y) - q_\lambda (y) | \leq \| q_\lambda - q_\lambda (y) \|_{L^\infty ( B_r (y) )}.
$$
Since $q_\lambda$ is continuous at $y$, sending $r \to 0^+$, we can get the desired result. 
\end{proof}

Here we state an estimate to  $G_\lambda(Y, \eta)$ whose proof is elementary and so is omitted.

\begin{lemma}\label{G3} For $\lambda < |Y| \leq 10 \lambda$ and $|\eta| > \lambda$, there exists $C > 0$ depending only on $n$ and $\sigma$ such that
\begin{equation}\label{C1GC2}
G_\lambda (Y, \eta) \leq C \frac{ (|Y| - \lambda) (|\eta|^2 - \lambda^2) }{ \lambda |Y - \eta|^{n - 2 \sigma + 2} }.
\end{equation}
\end{lemma}

Finally, we also need the following maximum principle whose proof can be found in \cite{JDSX}.

\begin{lemma}[\cite{JDSX}]\label{MAXPRI} Suppose $U \in W^{1, 2} (t^{1 - 2 \sigma}, \mathcal{B}_1^+ \setminus \overline{\mathcal{B}_\varepsilon^+}) \cap C(\overline{\mathcal{B}_1^+} \setminus \{ 0 \})$ for any $0 < \varepsilon < 1$ and
$$
\liminf\limits_{Y \to 0} U(Y) > - \infty.
$$
Suppose $U$ solves
$$
\left\{
\aligned
{\rm div} (t^{1 - 2 \sigma} \nabla U) & \leq 0 & \textmd{in} & ~ \mathcal{B}_1^+, \\
\frac{\partial U}{\partial \nu^\sigma} & \geq 0 ~~~~~~ & \textmd{on} & ~ \partial' \mathcal{B}_1^+ \setminus \{ 0 \},
\endaligned
\right.
$$
in the weak sense. Then
$$
U(Y) \geq \inf\limits_{\partial'' \mathcal{B}_1^+} U ~~~~~~ \textmd{for all} ~ Y \in \overline{\mathcal{B}_1^+} \setminus \{ 0 \}.
$$
\end{lemma}

\section{Local estimates in lower dimensions}\label{Sect3}

In this section, we shall prove Theorem \ref{Pri} using the moving sphere method introduced by Li-Zhu \cite{LZ-95}. 
 
\vskip0.10in  

\noindent {\bf Proof of Theorem \ref{Pri}.} Suppose by contradiction that there exists a sequence $\{ x_j \}_{j = 1}^\infty \subset B_1 \setminus \{ 0 \}$ such that
$$
|x_j| \to 0 ~~~~~~ \textmd{as} ~ j \to \infty,
$$
but
$$
|x_j|^\frac{n - 2 \sigma}{2} u(x_j) \to \infty ~~~~~~ \textmd{as} ~ j \to \infty.
$$

\vskip0.1in

{\it Step 1. We claim that $\{ x_j \}_{j = 1}^\infty$ can be chosen as the local maximum points of $u$.}

Consider
$$
f_j (x) : = \bigg( \frac{|x_j|}{2} - |x - x_j| \bigg)^\frac{n - 2 \sigma}{2} u(x) ~~~~~~ \textmd{for} ~ |x - x_j| \leq \frac{|x_j|}{2}.
$$
Since $u$ is positive and continuous in $\overline{B_{|x_j|/2} (x_j)}$, we can find a point $\bar{x}_j \in \overline{B_{|x_j|/2} (x_j)}$ such that
$$
f_j (\bar{x}_j) = \max\limits_{|x - x_j| \leq |x_j|/2} f_j (x) > 0.
$$
Let $2 \mu_j : = |x_j|/2 - |\bar{x}_j - x_j|$, then
$$
0 < 2 \mu_j \leq \frac{|x_j|}{2} ~~~~~~ \textmd{and} ~~~~~~ \frac{|x_j|}{2} - |x - x_j| \geq \mu_j ~~~ \forall ~ |x - \bar{x}_j| \leq \mu_j.
$$
By the definition of $f_j$, we have
$$
(2 \mu_j)^\frac{n - 2 \sigma}{2} u(\bar{x}_j) = f_j (\bar{x}_j) \geq f_j (x) \geq (\mu_j)^\frac{n - 2 \sigma}{2} u(x) ~~~~~~ \forall ~ |x - \bar{x}_j| \leq \mu_j.
$$
Hence we have
\begin{equation}\label{2uj-u}
2^\frac{n - 2 \sigma}{2} u(\bar{x}_j) \geq u(x) ~~~~~~ \forall ~ |x - \bar{x}_j| \leq \mu_j.
\end{equation}
We also have
\begin{equation}\label{2mujuj}
(2 \mu_j)^\frac{n - 2 \sigma}{2} u(\bar{x}_j) = f_j (\bar{x}_j) \geq f_j (x_j) = \left( \frac{|x_j|}{2} \right)^\frac{n - 2 \sigma}{2} u(x_j) \to \infty ~~~~~~ \textmd{as} ~ j \to \infty.
\end{equation}
Now we define
$$
\overline{W}_j (y, t) = \frac{1}{u(\bar{x}_j)} U \left( \bar{x}_j + \frac{y}{ u(\bar{x}_j)^\frac{2}{n - 2 \sigma} }, \frac{t}{ u(\bar{x}_j)^\frac{2}{n - 2 \sigma} } \right) ~~~~~~ \textmd{for} ~ (y, t) \in \Xi_j,
$$
where
$$
\Xi_j : = \left\{ (y, t) \in \R_+^{n + 1} : \left( \bar{x}_j + \frac{y}{ u(\bar{x}_j)^\frac{2}{n - 2 \sigma} }, \frac{t}{ u(\bar{x}_j)^\frac{2}{n - 2 \sigma} } \right) \in \mathcal{B}_1^+ \right\}.
$$
Let $\overline{w}_j (y) : = \overline{W}_j (y, 0)$, then $\overline{W}_j$ satisfies $\overline{w}_j (0) = 1$ and
$$
\left\{
\aligned
{\rm div} (t^{1 - 2 \sigma} \nabla \overline{W}_j) & = 0 & \textmd{in} & ~ \Xi_j, \\
\frac{\partial \overline{W}_j}{\partial \nu^\sigma} (y, 0) & = K \Big( \bar{x}_j + u(\bar{x}_j)^{ - \frac{2}{n - 2 \sigma} } y \Big) \overline{w}_j^\frac{n + 2 \sigma}{n - 2 \sigma} (y) ~~~~~~ & \textmd{on} & ~ \partial' \Xi_j \setminus \{ - u(\bar{x}_j)^\frac{2}{n - 2 \sigma} \bar{x}_j \}.
\endaligned
\right.
$$
Moreover, it follows from \eqref{2uj-u} and \eqref{2mujuj} that
$$
\overline{w}_j (y) \leq 2^\frac{n - 2 \sigma}{2} ~~~~~~ \textmd{in} ~ B_{R_j},
$$
where
\begin{equation}\label{Rj} 
R_j : =  \mu_j u(\bar{x}_j)^\frac{2}{n - 2 \sigma} \to \infty ~~~~~~ \textmd{as} ~ j \to \infty.
\end{equation}

By \cite[Proposition 2.6 ]{JLX-14},  for any given $\bar{t} > 0$ we have
$$
0 \leq W_j \leq C(\bar{t}) ~~~~~~ \textmd{in} ~ B_{R_j/2} \times [0, \bar{t}),
$$
where $C(\bar{t})$ depends only on $n$, $\sigma$, $\|K\|_{L^\infty(B_1)}$ and $\bar{t}$. Since $1/2 \leq \sigma < 1$, by the regularity results in \cite{Cabre,JLX-14},  we have that,   after passing to a subsequence, for some non-negative function $\overline{W} \in W_{loc}^{1, 2} (t^{1 - 2 \sigma}, \overline{\R_+^{n + 1}}) \cap C_{loc}^1 (\overline{\R_+^{n + 1}})$   
$$
\left\{
\aligned
\overline{W}_j & \rightharpoonup \overline{W} ~~~~~~ && \textmd{weakly in} ~ W_{loc}^{1, 2} (t^{1 - 2 \sigma}, \R_+^{n + 1}), \\
\overline{W}_j & \to \overline{W} && \textmd{in} ~ C_{loc}^1 (\overline{\R_+^{n + 1}}). 
\endaligned
\right.
$$
Denote $\overline{w} (y) : = \overline{W} (y, 0)$. Moreover, $\overline{W}$ satisfies
$$
\left\{
\aligned
{\rm div} (t^{1 - 2 \sigma} \nabla \overline{W}) & = 0 & \textmd{in} & ~ \R_+^{n + 1}, \\
\frac{\partial \overline{W}}{\partial \nu^\sigma} & = K(0) \overline{w}^\frac{n + 2 \sigma}{n - 2 \sigma} ~~~~~~ & \textmd{on} & ~ \partial \R_+^{n + 1},
\endaligned
\right.
$$
and $\overline{w} (0) = 1$. Without loss of generality, we may assume $K(0) = \widetilde{C}_{n, \sigma}$ which is defined in \eqref{Cns}.  By the Liouville theorem in \cite{JLX-14}, we have
$$
\overline{w} (y) = \left( \frac{\mu}{1 + \mu^2 |y - y_0|^2} \right)^\frac{n - 2 \sigma}{2},
$$
for some $y_0 \in \R^n$ and $\mu \geq 1$. Hence $\overline{w}$ has an absolute maximum value $\mu^\frac{n - 2 \sigma}{2}$ at $y_0$. It implies that $\overline{w}_j$ has a local maximum at a point $y_j$ close to $y_0$ and $\overline{w}_j (y_j) \geq \mu^\frac{n - 2 \sigma}{2}/2$ when $j$ is large. Let 
$$
\tilde{x}_j : = \bar{x}_j + u(\bar{x}_j)^{ - \frac{2}{n - 2 \sigma} } y_j,
$$
then $\{ \tilde{x}_j \}_{j = 1}^\infty$ are local maximum points of $u$ for large $j$ and
$$
u( \tilde{x}_j ) = u( \bar{x}_j ) \overline{w}_j (y_j) \geq \frac{ \mu^\frac{n - 2 \sigma}{2} }{2} u(\bar{x}_j).
$$
It follows from \eqref{Rj} that $|\tilde{x}_j - \bar{x}_j| = u(\bar{x}_j)^{ - \frac{2}{n - 2 \sigma} } |y_j| < \mu_j$ when $j$ is large, then $\tilde{x}_j \in B_{\mu_j} (\bar{x}_j) \subset B_{|x_j|/2} (x_j)$. Thus we have $|\tilde{x}_j| \geq |x_j|/2 \geq 2 \mu_j$ when $j$ is large. By \eqref{Rj} we can obtain
$$
|\tilde{x}_j|^\frac{n - 2 \sigma}{2} u(\tilde{x}_j) \geq \frac{ (2 \mu)^\frac{n - 2 \sigma}{2} }{2} (\mu_j)^\frac{n - 2 \sigma}{2} u(\bar{x}_j) \to \infty ~~~~~~ \textmd{as} ~ j \to \infty.
$$

Using the same arguments as before,  after passing to a subsequence we can prove that the function
$$
W_j (y, t) : = \frac{1}{u(\tilde{x}_j)} U \left( \tilde{x}_j + \frac{y}{ u(\tilde{x}_j)^\frac{2}{n - 2 \sigma} }, \frac{t}{ u(\tilde{x}_j)^\frac{2}{n - 2 \sigma} } \right)
$$
converges to $\widetilde{W}$ in $C_{loc}^1 (\overline{\R_+^{n + 1}} )$ and weakly in $W_{loc}^{1, 2} (t^{1 - 2 \sigma}, \R_+^{n + 1})$ which solves
$$
\left\{
\aligned
{\rm div} (t^{1 - 2 \sigma} \nabla \widetilde{W}) & = 0 & \textmd{in} & ~ \R_+^{n + 1}, \\
\frac{\partial \widetilde{W}}{\partial \nu^\sigma} & = K(0) \widetilde{w}^\frac{n + 2 \sigma}{n - 2 \sigma} ~~~~~~ & \textmd{on} & ~ \partial \R_+^{n + 1}, \\
\max_{\R^n} \widetilde{w} & = \widetilde{w} (0) = 1,
\endaligned
\right.
$$
where $\widetilde{w} (y) : = \widetilde{W} (y, 0)$. By the Liouville theorem in \cite{JLX-14},
$$
\widetilde{w} (y) = \left( \frac{1}{1 + |y|^2} \right)^\frac{n - 2 \sigma}{2} ~~~~ \textmd{and}  ~~~~ \widetilde{W} = \mathcal{P}_\sigma [\widetilde{w}], 
$$
where $\mathcal{P}_\sigma$ is the Poisson kernel given in \eqref{Poss09}.  From now on we consider $x_j$ as $\tilde{x}_j$. The claim is proved.

\vskip0.1in

{\it Step 2. We give some estimates for the difference between $W_j$ and its Kelvin transformation. }

Let $L_j : = u(x_j)$ and
$$
\Omega_j : = \left\{ (y, t) \in \R_+^{n + 1} : \Big( x_j + L_j^{ - \frac{2}{n - 2 \sigma} } y, L_j^{ - \frac{2}{n - 2 \sigma} } t \Big) \in \mathcal{B}_1^+ \right\}.
$$
Then for large $j$, we know that $\mathcal{B}_5^+ \subset \Omega_j$. For $\lambda \in [1/2, 2]$, let
$$
W_j^\lambda (Y) : = \left( \frac{\lambda}{|Y|} \right)^{n - 2 \sigma} W_j \left( \frac{\lambda^2 Y}{|Y|^2} \right) ~~~~~~ \textmd{for} ~ Y \in \overline{\Omega_j} \setminus \mathcal{B}_\lambda^+.
$$

Now we define $W_\lambda (Y) : = W_j (Y) - W_j^\lambda (Y)$ and $w_\lambda (y) : = W_\lambda (y, 0)$. Here we omit $j$ for simplicity.  We also denote 
$$
K_j (y) : = K \Big( x_j + L_j^{ - \frac{2}{n - 2 \sigma} } y \Big). 
$$
Then
\begin{equation}\label{WLE}
\left\{
\aligned
{\rm div} (t^{1 - 2 \sigma} \nabla W_\lambda) & = 0 & \textmd{in} & ~ \Omega_j \setminus \overline{\mathcal{B}_\lambda^+}, \\
\frac{\partial W_\lambda}{\partial \nu^\sigma} & = b_\lambda w_\lambda - q_\lambda ~~~~~~ & \textmd{on} & ~ \partial' (\Omega_j \setminus \mathcal{B}_\lambda^+) \setminus \{ - L_j^\frac{2}{n - 2 \sigma} x_j \},
\endaligned
\right.
\end{equation}
where
\begin{equation}\label{bLP}
b_\lambda (y) : = K_j (y) \frac{ w_j (y)^\frac{n + 2 \sigma}{n - 2 \sigma} - w_j^\lambda (y)^\frac{n + 2 \sigma}{n - 2 \sigma} }{ w_j (y) - w_j^\lambda (y) } \geq 0 
\end{equation}
and
\begin{equation}\label{qL1}
q_\lambda (y) : = ( K_j (y^\lambda) - K_j (y) ) w_j^\lambda (y)^\frac{n + 2 \sigma}{n - 2 \sigma}.
\end{equation}
Notice that $q_\lambda \in C(\overline{\partial' \Omega_j} \setminus B_\lambda)$. Moreover, we have the following estimates for the difference between $W_j$ and  $ W_j^{\lambda}$ when $\lambda=1/2$ or $2$.

\begin{lemma}\label{W-WL} 
Let $\lambda_0 = 1/2$ and $\lambda_1 = 2$.  There exists $c_0 > 0$ and $j_0 > 1$ such that for all $j \geq j_0$,
\begin{equation}\label{W-WL0}
W_j (Y) - W_j^{\lambda_0} (Y) \geq c_0 (|Y| - \lambda_0) |Y|^{2 \sigma - n - 1} + c_0 L_j^{- 1} (\lambda_0^{2 \sigma - n} - |Y|^{2 \sigma - n})
\end{equation}
for all $Y \in \Omega_j \setminus \mathcal{B}_{\lambda_0}^+$. Moreover, there exists $Y^* \in \overline{\mathcal{B}_{2 \lambda_1}^+} \setminus \overline{\mathcal{B}_{\lambda_1}^+}$ such that for all $j \geq j_0$,
\begin{equation}\label{W-WL1}
W_j (Y^*) - W_j^{\lambda_1} (Y^*) \leq - c_0.
\end{equation}
\end{lemma}

\begin{proof} Since $\widetilde{W} (Y) = |Y|^{2 \sigma - n} \widetilde{W} (Y/|Y|^2)$, we have $\lim_{|Y| \to \infty} |Y|^{n - 2 \sigma} \widetilde{W} (Y) = \widetilde{W} (0) = 1$. Similarly, we also have $\lim_{|Y| \to \infty} |Y|^{n - 2 \sigma} \widetilde{W}^{\lambda_0} (Y) = \lambda_0^{n - 2 \sigma}$. Hence, there exist a small $0 < c_1 < 1/4$ and a large $R > 4$ such that
$$
|\widetilde{W} (Y) - |Y|^{2 \sigma - n}| \leq \frac{c_1}{2} |Y|^{2 \sigma - n} ~~~~~~ \textmd{for} ~ |Y| \geq R,
$$
and
$$
\widetilde{W}^{\lambda_0} (Y) \leq (1 - 3 c_1) |Y|^{2 \sigma - n} ~~~~~~ \textmd{for} ~ |Y| \geq R.
$$
Since $W_j$ converges to $\widetilde{W}$ in $C_{loc}^1 (\overline{\R_+^{n + 1}})$, for large $j$ we obtain
\begin{equation}\label{WjR}
W_j (Y) \geq (1 - c_1) |Y|^{2 \sigma - n} ~~~~~~ \textmd{for} ~ |Y| = R, 
\end{equation}
and
$$
W_j^{\lambda_0} (Y) \leq (1 - 2 c_1) |Y|^{2 \sigma - n} ~~~~~~ \textmd{for} ~ |Y| \geq R.
$$
Recall that for $Y \in \partial'' \Omega_j$,
$$
|Y| = L_j^\frac{2}{n - 2 \sigma} \Big| L_j^{ - \frac{2}{n - 2 \sigma} } Y \Big| \geq L_j^\frac{2}{n - 2 \sigma} (1 - |x_j|).
$$
Thus
\begin{equation}\label{YWinf}
|Y|^{n - 2 \sigma} W_j (Y) \geq L_j (1 - |x_j|)^{n - 2 \sigma} \inf_{ \overline{\mathcal{B}_1^+} \setminus \{ 0 \} } U \to \infty ~~~~~~ \textmd{as} ~ j \to \infty.
\end{equation}
It follows from \eqref{WjR}, \eqref{YWinf} and Lemma \ref{MAXPRI} that for sufficiently large $j$, we have
$$
W_j (Y) \geq (1 - c_1) |Y|^{2 \sigma - n} ~~~~~~ \textmd{for} ~ Y \in \Omega_j \setminus \mathcal{B}_R^+.
$$
Again, since $W_j$ converges to $\widetilde{W}$ in $C_{loc}^1 (\overline{\R_+^{n + 1}})$, by Lemma \ref{B-BL}, there exists $c_2 > 0$ such that
\begin{equation}\label{W-WLe2}
W_j (Y) - W_j^{\lambda_0} (Y) \geq c_2 (|Y| - \lambda_0) |Y|^{2 \sigma - n - 1} ~~~~~~ \textmd{for} ~ Y \in \mathcal{B}_{R}^+ \setminus \mathcal{B}_{\lambda_0}^+
\end{equation}
for sufficiently large $j$.

Now we show that \eqref{W-WL0} holds for $Y \in \Omega_j \setminus \mathcal{B}_R^+$. By the definition of $W_j$, there exists $c_3 > 0$ such that
$$
W_j (Y) \geq L_j^{- 1} \inf_{ \overline{\mathcal{B}_1^+} \setminus \{ 0 \} } U \geq c_3 L_j^{- 1} (\lambda_0^{2 \sigma - n} - |Y|^{2 \sigma - n}) ~~~~~~ \textmd{for} ~ Y \in \Omega_j \setminus \mathcal{B}_R^+.
$$
Then we have
$$
\aligned
W_j (Y) - W_j^{\lambda_0} (Y) & = \frac{c_1}{2} W_j (Y) + \left( 1 - \frac{c_1}{2} \right) W_j (Y) - W_j^{\lambda_0} (Y) \\
& \geq \frac{c_1 c_3}{2} L_j^{- 1} (\lambda_0^{2 \sigma - n} - |Y|^{2 \sigma - n}) + \left[ \left(1 - \frac{c_1}{2} \right) (1 - c_1) - (1 - 2 c_1) \right] |Y|^{2 \sigma - n} \\
& \geq \frac{c_1 c_3}{2} L_i^{- 1} (\lambda_0^{2 \sigma - n} - |Y|^{2 \sigma - n}) + \frac{c_1 (1 + c_1)}{2} |Y|^{2 \sigma - n}
\endaligned
$$
for $Y \in \Omega_j \setminus \mathcal{B}_R^+$. Therefore, there exists a small $c_0>0$ such that for large $j$,
$$
W_j (Y) - W_j^{\lambda_0} (Y) \geq c_0 |Y|^{2 \sigma - n} + c_0 L_j^{- 1} (\lambda_0^{2 \sigma - n} - |Y|^{2 \sigma - n}) ~~~~~~ \textmd{for all} ~ Y \in \Omega_j \setminus \mathcal{B}_R^+.
$$
This together with \eqref{W-WLe2} implies that \eqref{W-WL0} holds by choosing $c_0$ sufficiently small. Finally, \eqref{W-WL1} is a direct consequence of Lemma \ref{B-BL}.
\end{proof}

Let $\tau \in (0, 1/4)$ be a constant to be specified later, we define
$$
\Pi_j = \left\{ (y, t) \in \R_+^{n + 1} : \Big( x_j + L_j^{ - \frac{2}{n - 2 \sigma} } y, L_j^{ - \frac{2}{n - 2 \sigma} } t \Big) \in \mathcal{B}_{2 \tau}^+ \right\} \subset \Omega_j 
$$
and 
$$
\Sigma_j = \left\{ (y, t) \in \R_+^{n + 1} : \Big( x_j + L_j^{ - \frac{2}{n - 2 \sigma} } y, L_j^{ - \frac{2}{n - 2 \sigma} } t \Big) \in \mathcal{B}_\tau^+ \right\} \subset \Pi_j.
$$
Define 
\begin{equation}\label{Phi0123}
\Phi_\lambda (Y) : = \int_{\partial' \Pi_j \setminus B_\lambda} G_\lambda (Y, \eta) q_\lambda (\eta) d \eta ~~~~~~ \textmd{for} ~ Y \in \R_+^{n + 1} \setminus \mathcal{B}_\lambda^+, 
\end{equation} 
where $G_\lambda (Y, \eta)$ is the Green's function defined in \eqref{Green} and $q_\lambda$ is the function given in \eqref{qL1}.    From Lemma \ref{G2}, we know that $\Phi_\lambda (Y) = 0$ when $|Y| = \lambda$ and
\begin{equation}\label{PLE}
\left\{
\aligned
{\rm div} (t^{1 - 2 \sigma} \nabla \Phi_\lambda) & = 0 & \textmd{in} & ~ \Pi_j \setminus \overline{\mathcal{B}_\lambda^+}, \\
\frac{\partial \Phi_\lambda}{\partial \nu^\sigma} & = q_\lambda ~~~~~~ & \textmd{on} & ~ \partial' \Pi_j \setminus \overline{B_\lambda}.
\endaligned
\right.
\end{equation}

\vskip0.1in

{\it Step 3. We state some estimates for $\Phi_\lambda$.}  

Since $w_j$ converges to $\widetilde{w}$ in $C_{loc}^1 (\R^n)$, we have for any $\lambda \in [1/2, 2]$,
\begin{equation}\label{wleta}
0 \leq w_j^\lambda (\eta) = \left( \frac{\lambda}{|\eta|} \right)^{n - 2 \sigma} w_j \left( \frac{\lambda^2 \eta}{|\eta|^2} \right) \leq C |\eta|^{2 \sigma - n} ~~~~~~ \textmd{for} ~ \eta \in \partial' \Pi_j \setminus B_\lambda
\end{equation}
when $j$ is sufficiently large, where $C > 0$ depends only on $n$ and $\sigma$.

By our assumptions on $K$, we have
\begin{equation}\label{KL}
|K_j (\eta^\lambda) - K_j (\eta)| \leq c(\tau) L_j^{- 1} (|\eta| - \lambda)^\alpha ~~~~~~ \textmd{for} ~ \eta \in \partial' \Pi_j \setminus B_\lambda,
\end{equation}
where $c(\cdot)$ is a non-negative function with $c(0) = 0$. This together with \eqref{qL1} and \eqref{wleta} implies that
\begin{equation}\label{qL}
|q_\lambda (\eta)| \leq c(\tau) L_j^{- 1} (|\eta| - \lambda)^\alpha |\eta|^{- n - 2 \sigma} ~~~~~~  \textmd{for} ~ \eta \in \partial' \Pi_j \setminus B_\lambda,
\end{equation}
Consequently, we have the following estimates for $\Phi_\lambda$.

\begin{lemma}\label{PhiL} For any $\lambda \in [1/2, 2]$, we have
$$
|\Phi_\lambda (Y)| \leq
\left\{
\aligned
& c(\tau) L_j^{- 1} (|Y| - \lambda) ~~~~~~ & \textmd{if} & ~ Y \in \mathcal{B}_{4 \lambda}^+ \setminus \mathcal{B}_\lambda^+, \\
& c(\tau) L_j^{- 1} |Y|^{2 \sigma - n} \log |Y| ~~~~~~ & \textmd{if} & ~ Y \in \Sigma_j \setminus \mathcal{B}_{4 \lambda}^+,
\endaligned
\right.
$$
where $c(\cdot)$ is a non-negative function with $c(0) = 0$.
\end{lemma}

\begin{proof} It follows from \eqref{qL} that for $Y \in \Sigma_j \setminus \mathcal{B}_\lambda^+$,
$$
|\Phi_\lambda (Y)| \leq c(\tau) L_j^{- 1} \int_{\partial' \Pi_j \setminus B_\lambda} G_\lambda (Y, \eta) (|\eta| - \lambda)^\alpha |\eta|^{- n - 2 \sigma} d \eta.
$$
We have two cases:

\vskip0.1in

{\it Case 1. $\lambda < |Y| < 4 \lambda$.}
Let $\partial' \Pi_j \setminus B_\lambda = A \cup B \cup D$ where
$$
\aligned
A & : = \{ \eta \in \partial' \Pi_j \setminus B_\lambda : |Y - \eta| < (|Y| - \lambda)/3 \}, \\
B & : = \{ \eta \in \partial' \Pi_j \setminus B_\lambda : |Y - \eta| \geq (|Y| - \lambda)/3 ~ \textmd{and} ~ |\eta| \leq 8 \lambda \}, \\
D & : = \{ \eta \in \partial' \Pi_j \setminus B_\lambda : |\eta| > 8 \lambda \}.
\endaligned
$$
Since $G_\lambda (Y, \eta) \leq C |Y - \eta|^{2 \sigma - n}$, $2 (|Y| - \lambda)/3 \leq |\eta| - \lambda \leq 4 (|Y| - \lambda)/3$ and $\lambda \leq |\eta| \leq 5 \lambda$ for any $\eta \in A$, we have 
$$
\aligned
\int_A G_\lambda (Y, \eta) (|\eta| - \lambda)^\alpha |\eta|^{- n - 2 \sigma} d \eta & \leq C (|Y| - \lambda)^\alpha \int_A |Y - \eta|^{2 \sigma - n} d \eta \\
& \leq C (|Y| - \lambda)^\alpha \int_{ \left\{ \eta \in \partial' \Pi_j \setminus B_\lambda : |y - \eta| < \frac{|Y| - \lambda}{3} \right\} } |y - \eta|^{2 \sigma - n} d \eta \\
& \leq C (|Y| - \lambda)^{\alpha + 2 \sigma} \leq C (|Y| - \lambda).
\endaligned
$$
By Lemma \ref{G3} and $|\eta| - \lambda \leq 4 |Y - \eta|$ for any $\eta \in B$, we have 
$$
\aligned
\int_B G_\lambda (Y, \eta) (|\eta| - \lambda)^\alpha |\eta|^{- n - 2 \sigma} d \eta & \leq C (|Y| - \lambda) \int_B |Y - \eta|^{2 \sigma - n - 2} (|\eta| - \lambda)^{1 + \alpha} d \eta \\
& \leq C (|Y| - \lambda) \int_B |Y - \eta|^{2 \sigma - n - 1 + \alpha} d \eta \\
& \leq C (|Y| - \lambda).
\endaligned
$$
For any $\eta \in D$, we have $|Y - \eta| \geq |\eta| - |Y| \geq |\eta|/2$ and $7 |\eta|/8 \leq |\eta| - \lambda \leq |\eta|$. By Lemma \ref{G3},
$$
\int_D G_\lambda (Y, \eta) (|\eta| - \lambda)^\alpha |\eta|^{- n - 2 \sigma} d \eta \leq C (|Y| - \lambda) \int_D |\eta|^{\alpha - 2 n} d \eta \leq C (|Y| - \lambda).
$$
Consequently, we obtain the first estimate in Lemma \ref{PhiL}.

\vskip0.1in

{\it Case 2. $Y \in \Sigma_j \setminus \mathcal{B}_{4 \lambda}^+$.} Let $\partial' \Pi_j \setminus B_\lambda = A_1 \cup A_2 \cup A_3 \cup A_4$ where
$$
\aligned
A_1 & : = \{ \eta \in \partial' \Pi_j \setminus B_\lambda : |\eta| < |Y|/2 \}, \\
A_2 & : = \{ \eta \in \partial' \Pi_j \setminus B_\lambda : |\eta| > 2 |Y| \}, \\
A_3 & : = \{ \eta \in \partial' \Pi_j \setminus B_\lambda : |Y - \eta| \leq |Y|/2 \}, \\
A_4 & : = \{ \eta \in \partial' \Pi_j \setminus B_\lambda : |Y - \eta| \geq |Y|/2 ~ \textmd{and} ~ |Y| \leq |\eta| \leq 2 |Y| \}.
\endaligned
$$
Since $G_\lambda (Y, \eta) \leq C |Y - \eta|^{2 \sigma - n}$ and $|Y - \eta| \geq |Y|/2$ for any $\eta \in A_1$,
$$
\aligned
\int_{A_1} G_\lambda (Y, \eta) (|\eta| - \lambda)^\alpha |\eta|^{- n - 2 \sigma} d \eta & \leq C \int_{A_1} |Y - \eta|^{2 \sigma - n} (|\eta| - \lambda)^\alpha |\eta|^{- n - 2 \sigma} d \eta \\
& \leq C  |Y|^{2 \sigma - n} \int_{A_1} |\eta|^{\alpha - n - 2 \sigma} d \eta \\
& \leq
\left\{
\aligned
& C |Y|^{2 \sigma - n} ~~~~~~ & \textmd{if} & ~ \alpha < 2 \sigma, \\
& C |Y|^{2 \sigma - n} \log |Y| ~~~~~~ & \textmd{if} & ~ \alpha = 2 \sigma. 
\endaligned
\right.
\endaligned
$$
For any $\eta \in A_2$, $|Y - \eta| \geq |\eta| - |Y| \geq |\eta|/2$, then 
$$
\aligned
\int_{A_2} G_\lambda (Y, \eta) (|\eta| - \lambda)^\alpha |\eta|^{- n - 2 \sigma} d \eta & \leq C \int_{A_2} |Y - \eta|^{2 \sigma - n} |\eta|^{\alpha - n - 2 \sigma} d \eta \\
& \leq C \int_{A_2} |\eta|^{\alpha - 2 n} d \eta \\
& \leq C |Y|^{\alpha - n}.
\endaligned
$$
For any $\eta \in A_3$, $|\eta| \geq |Y| - |Y - \eta| \geq |Y|/2$, then 
$$
\aligned
\int_{A_3} G_\lambda (Y, \eta) (|\eta| - \lambda)^\alpha |\eta|^{- n - 2 \sigma} d \eta & \leq C \int_{A_3} |Y - \eta|^{2 \sigma - n} |\eta|^{\alpha - n - 2 \sigma} d \eta \\
& \leq C |Y|^{\alpha - n - 2 \sigma} \int_{A_3} |Y - \eta|^{2 \sigma - n} d \eta \\
& \leq C |Y|^{\alpha - n - 2 \sigma} \int_{ \left\{ \eta \in \partial' \Pi_j \setminus B_\lambda : |y - \eta| \leq \frac{|Y|}{2} \right\} } |y - \eta|^{2 \sigma - n} d \eta \\
& \leq C |Y|^{\alpha - n}.
\endaligned
$$
We also have
$$
\aligned
\int_{A_4} G_\lambda (Y, \eta) (|\eta| - \lambda)^\alpha |\eta|^{- n - 2 \sigma} d \eta & \leq C \int_{A_4} |Y - \eta|^{2 \sigma - n} |\eta|^{\alpha - n - 2 \sigma} d \eta \\
& \leq C |Y|^{2 \sigma - n} \int_{A_4} |\eta|^{\alpha - n - 2 \sigma} d \eta \\
& \leq C |Y|^{\alpha - n}.
\endaligned
$$
From above, we get the second estimate in Lemma \ref{PhiL}.
\end{proof}

{\it Step 4. We will use the method of moving spheres to reach a contradiction.}

Inspired by \cite{CL-97,TZ-06,Zhang-02},  for $\lambda \in [1/2, 2]$ we define 
$$
A_\lambda (Y) : = - c_4 L_j^{- 1} (\lambda^{2 \sigma - n} - |Y|^{2 \sigma - n}) + \Phi_\lambda (Y),  ~~~~ ~ Y \in \R_+^{n + 1} \setminus \mathcal{B}_\lambda^+, 
$$
where $\Phi_\lambda$ is defined as in \eqref{Phi0123} and the constant $c_4$ is given by
$$
c_4 : = \frac{1}{32} \min\left\{ c_0, \inf_{ \overline{\mathcal{B}_1^+} \setminus \{ 0 \} } U \right\} >0.
$$
Here we recall that $c_0>0$ is defined in Lemma \ref{W-WL}.   

By Lemma \ref{PhiL}, we can choose $\tau$ sufficiently small such that for any $\lambda \in [1/2, 2]$,
\begin{equation}\label{ALNG}
A_\lambda (Y) < 0 ~~~~~~ \textmd{for all} ~ Y \in \Sigma_j \setminus \overline{\mathcal{B}_\lambda^+} 
\end{equation}
when $j$  is sufficiently large.  Then for $Y \in \partial'' \Sigma_j$, we have 
$$
\frac{ L_j^\frac{2}{n - 2 \sigma} \tau }{2} \leq L_j^\frac{2}{n - 2 \sigma} (\tau - |x_j|) \leq |Y| \leq L_j^\frac{2}{n - 2 \sigma} (\tau + |x_j|) \leq 2 L_j^\frac{2}{n - 2 \sigma} \tau
$$
for large $j$. By Lemma \ref{PhiL} and \eqref{ALNG},
\begin{equation}\label{AL-WL}
|A_\lambda (Y)| \leq c_4 \lambda^{2 \sigma - n} L_j^{- 1} \leq 2^{n - 2 \sigma} c_4 L_j^{- 1} \leq 8 c_4 L_j^{- 1} ~~~~~~ \textmd{for} ~ Y \in \partial'' \Sigma_j,
\end{equation}
and
\begin{equation}
\aligned
W_\lambda (Y) & = W_j (Y) - W_j^\lambda (Y) \\
& \geq L_j^{- 1} \inf_{ \overline{\mathcal{B}_1^+} \setminus \{ 0 \} } U - \lambda^{n - 2 \sigma} |Y|^{2 \sigma - n} W_j \left( \frac{\lambda^2 Y}{|Y|^2} \right) \\
& \geq L_j^{- 1} \inf_{ \overline{\mathcal{B}_1^+} \setminus \{ 0 \} } U - 4^{n - 2 \sigma} \tau^{2 \sigma - n} \| W_j \|_{L^\infty (\mathcal{B}_2^+)} L_j^{- 2} \\
& \geq L_j^{- 1} \inf_{ \overline{\mathcal{B}_1^+} \setminus \{ 0 \} } U - 8^{n - 2 \sigma} \tau^{2 \sigma - n} \| \widetilde{W} \|_{L^\infty (\mathcal{B}_2^+)} L_j^{- 2} \geq 16 c_4 L_j^{- 1}
\endaligned
\end{equation}
for $Y \in \partial'' \Sigma_j$ and large $j$. It implies that
\begin{equation}\label{WApa}
(W_\lambda + A_\lambda) (Y) \geq 16 c_4 L_j^{- 1} ~~~~~~ \textmd{for} ~ Y \in \partial'' \Sigma_j 
\end{equation}
when $j$ is sufficiently large.

Now,  from \eqref{WLE},  \eqref{bLP}, \eqref{PLE} and \eqref{ALNG} we have 
\begin{equation}\label{mciuye5}
\left\{
\aligned
{\rm div} (t^{1 - 2 \sigma} \nabla (W_\lambda + A_\lambda)) & = 0 & \textmd{in} & ~ \Sigma_j \setminus \overline{\mathcal{B}_\lambda^+}, \\
\frac{\partial (W_\lambda + A_\lambda)}{\partial \nu^\sigma} & \geq b_\lambda (w_\lambda + a_\lambda) ~~~~~~ & \textmd{on} & ~ \partial' (\Sigma_j \setminus \mathcal{B}_\lambda^+) \setminus \{ - L_j^\frac{2}{n - 2 \sigma} x_j \},
\endaligned
\right.
\end{equation} 
where $w_\lambda(x)=W_\lambda(x, 0)$ and $a_\lambda(x)=A_\lambda(x, 0)$. 

For $\lambda_0=1/2$ and $\lambda_1 =2$,  we define 
$$
\lambda^* : = \sup \Big\{ \lambda \geq \lambda_0 : W_\mu + A_\mu \geq 0 ~ \textmd{in} ~ (\overline{\Sigma_j} \setminus \mathcal{B}_\mu^+) \setminus \{ - L_j^\frac{2}{n - 2 \sigma} x_j \} ~ \textmd{for all} ~ \lambda_0 \leq \mu \leq \lambda \Big\}.
$$
Then $\lambda^*$ is well-defined as $\lambda_0$ belongs to the above set by Lemma \ref{W-WL}.  From \eqref{W-WL1} and \eqref{ALNG}, we know that $\lambda^* < \lambda_1$. By continuity we have 
$$
W_{\lambda^*} + A_{\lambda^*} \geq 0 ~~~~~~~ \textmd{in} (\overline{\Sigma_j} \setminus \mathcal{B}_{\lambda^*}^+) \setminus \{ - L_j^\frac{2}{n - 2 \sigma} x_j \}.
$$
We also have $W_{\lambda^*} + A_{\lambda^*} = 0$ on $\partial'' \mathcal{B}_{\lambda^*}$. It follows from \eqref{WApa} and the strong maximum principle that $W_{\lambda^*} + A_{\lambda^*} > 0$ in $(\Sigma_j \setminus \overline{\mathcal{B}_{\lambda^*}^+}) \setminus \{ - L_j^\frac{2}{n - 2 \sigma} x_j \}$.

For $\delta > 0$ small which will be fixed later,  by Lemma \ref{MAXPRI}, there exists $c = c(\delta) > 0$ such that
$$
W_{\lambda^*} + A_{\lambda^*} > c ~~~~~~ \textmd{in} ~ (\overline{\Sigma_j} \setminus \mathcal{B}_{\lambda^* + \delta}^+) \setminus \{ - L_j^\frac{2}{n - 2 \sigma} x_j \}.
$$
By the uniform continuity, there exists $0 < \varepsilon = \varepsilon (\delta) < \delta$ small such that for all $\lambda^* < \lambda < \lambda^* + \varepsilon$,
$$
(W_\lambda + A_\lambda) - (W_{\lambda^*} + A_{\lambda^*}) > - c/2 ~~~~~~ \textmd{in} ~ (\overline{\Sigma_j} \setminus \mathcal{B}_{\lambda^* + \delta}^+) \setminus \{ - L_j^\frac{2}{n - 2 \sigma} x_j \}.
$$
Hence we obtain
$$
W_\lambda + A_\lambda > c/2 ~~~~~~ \textmd{in} ~ (\overline{\Sigma_j} \setminus \mathcal{B}_{\lambda^* + \delta}^+) \setminus \{ - L_j^\frac{2}{n - 2 \sigma} x_j \}.
$$
Next we will make use of the narrow domain technique from \cite{BN}.   Using $(W_\lambda + A_\lambda)^- : = \max \{ - (W_\lambda + A_\lambda), 0 \}$ as a test function in \eqref{mciuye5},   we obtain that for any $\lambda^* < \lambda < \lambda^* + \varepsilon$, 
$$
\aligned
& \int_{\mathcal{B}_{\lambda^* + \delta}^+ \setminus \mathcal{B}_\lambda^+} t^{1 - 2 \sigma} | \nabla (W_\lambda + A_\lambda)^- |^2 \\
& \leq - \int_{B_{\lambda^* + \delta} \setminus B_\lambda} b_\lambda (w_\lambda + a_\lambda) (w_\lambda + a_\lambda)^- \\
& \leq C \int_{B_{\lambda^* + \delta} \setminus B_\lambda} \big( (w_\lambda + a_\lambda)^- \big)^2 \\
& \leq C \bigg( \int_{B_{\lambda^* + \delta} \setminus B_\lambda} \big( (w_\lambda + a_\lambda)^- \big)^\frac{2 n}{n - 2 \sigma} \bigg)^\frac{n - 2 \sigma}{n} \mathcal{L}^n (B_{\lambda^* + \delta} \setminus B_\lambda)^\frac{2 \sigma}{n} \\
& \leq C \bigg( \int_{\mathcal{B}_{\lambda^* + \delta}^+ \setminus \mathcal{B}_\lambda^+} t^{1 - 2 \sigma} | \nabla (W_\lambda + A_\lambda)^- |^2 \bigg) \mathcal{L}^n (B_{\lambda^* + \delta} \setminus B_\lambda)^\frac{2 \sigma}{n}.
\endaligned
$$
We can fix $\delta > 0$ sufficiently small such that for all $\lambda^* < \lambda < \lambda^* + \varepsilon$,
$$
C \mathcal{L}^n (B_{\lambda^* + \delta} \setminus B_\lambda)^\frac{2 \sigma}{n} \leq 1/2.
$$
Then
$$
\nabla (W_\lambda + A_\lambda)^- = 0 ~~~~~~ \textmd{in} ~ \mathcal{B}_{\lambda^* + \delta}^+ \setminus \mathcal{B}_\lambda^+.
$$
Recall that
$$
(W_\lambda + A_\lambda)^- = 0 ~~~~~~ \textmd{on} ~ \partial'' \mathcal{B}_\lambda^+,
$$
then we have
$$
W_\lambda + A_\lambda \geq 0 ~~~~~~ \textmd{in} ~ \mathcal{B}_{\lambda^* + \delta}^+ \setminus \mathcal{B}_\lambda^+.
$$
This implies that there exists $\varepsilon > 0$ such that for all $\lambda^* < \lambda < \lambda^* + \varepsilon$,
$$
W_\lambda + A_\lambda \geq 0 ~~~~~~ \textmd{in} ~ (\overline{\Sigma_j} \setminus \mathcal{B}_{ \lambda}^+) \setminus \{ - L_j^\frac{2}{n - 2 \sigma} x_j \},
$$
which contradicts the definition of $\lambda^*$. The proof of Theorem \ref{Pri} is completed.  \hfill$\square$

\section{Large singular solutions in higher dimensions}\label{Sect4}

In this section, we shall prove Theorem \ref{Lar} and always assume $\sigma \in (0, 1)$. We write $a_i \sim b_i$ if the sequence $\{ a_i/b_i \}_{i = 1}^\infty$ is bounded between two positive constants depending only on $n$, $\sigma$, $\inf_{\R^n} k$ and $\sup_{\R^n} k$.

The definition of fractional Laplacian \eqref{FL} is equivalent to
\begin{equation}\label{FLS}
(- \Delta)^\sigma u(x) = \frac{ C_{n, \sigma} }{2} \lim\limits_{r \to 0^+} \int_{\R^n \setminus B_r} \frac{2 u(x) - u(x + y) - u(x - y)}{ |y|^{n + 2 \sigma} } dy.
\end{equation}
For real number $\gamma > 0$, we denote $C^\gamma$ as the space $C^{[\gamma], \gamma'}$ where $[\gamma]$ is the greatest integer satisfying $[\gamma]  < \gamma$ and $\gamma' = \gamma - [\gamma]$. Define
$$
L_\sigma (\R^n) = \left\{ u \in L_{loc}^1 (\R^n) : \int_{\R^n} \frac{|u(x)|}{ 1 + |x|^{n + 2 \sigma} } dx < \infty \right\}.
$$
If $u \in L_\sigma (\R^n) \cap C^{2 \sigma + \varepsilon} (\R^n \setminus \{ 0 \})$ for some $\varepsilon > 0$, then $(- \Delta)^\sigma u (x)$ is well-defined for every $x\in \R^n \setminus \{ 0 \}$  (see, e.g., \cite{Si-07}).      To prove Theorem \ref{Lar}, we also need the following simple lemma. 

\begin{lemma}[\cite{Ta-05}]\label{Talia}
Suppose $p > 1$, $\{ a_i \}_{i = 1}^N \subset (0, \infty)$, and $a_1 \geq a_i$ for $2 \leq i \leq N$. Then
$$
\frac{ \sum_{i = 1}^N a_i^p}{ \Big( \sum_{i = 1}^N a_i \Big)^p } \leq \frac{ 1 + \frac{a_2}{a_1} }{ 1 + p \frac{a_2}{a_1} } < 1.
$$
\end{lemma}

The method to prove Theorem \ref{Lar} is similar to that of Taliaferro \cite{Ta-05} for the Laplacian case $\sigma=1$, which has also been used to higher order conformal $Q$-curvature equation in our recent paper \cite{DYla}.   Here we will first set up a framework to fit the nonlocal equation 
\begin{equation}\label{Fravko9e7}
(- \Delta)^\sigma u = K(x) u^\frac{n + 2 \sigma}{n - 2 \sigma} ~~~~~~ \textmd{in} ~ \R^n \setminus \{ 0 \}.  
\end{equation} 
Roughly speaking,  we would construct a sequence of bubbles $\{ u_i \}_{i = 1}^\infty$ such that the function $\overline{u} : = \sum_{i = 1}^\infty u_i$ satisfies $\overline{u} (x) \neq \varphi (|x|)$ near the origin. Then we find a positive bounded function $u_0 \in C^{2 \sigma + \gamma} (\R^n \setminus \{ 0 \})$ by the method of sub- and super-solutions such that 
$$
u : = u_0 + \overline{u} ~~~~~~ \textmd{and} ~~~~~~ K : = \frac{ (- \Delta)^\sigma u }{ u^{ (n + 2 \sigma)/(n - 2 \sigma) } } 
$$
satisfy the corresponding conclusions of Theorem \ref{Lar} for the nonlocal equation \eqref{Fravko9e7}.  We point out that the bubbles $\{ u_i \}_{i = 1}^\infty$ will be selected very carefully so that $| \nabla K |$ can be continuous in the whole $\R^n$.  Eventually, we define $U : = \mathcal{P}_\sigma [u]$ which is the desired solution for our Theorem \ref{Lar}, where $\mathcal{P}_\sigma$ is the Poisson kernel given in \eqref{Poss09}.  The same arguments as in \cite{DYla,Ta-05}  will be omitted.  Meanwhile, we give the detailed proofs for those different with \cite{DYla,Ta-05}. In particular,  the main differences are as follows.

\begin{itemize}
\item We use the method of sub- and super-solutions for fractional Laplacian.

\item We show that $u_0$ is a distributional solution in $\R^n$ for a nonlocal equation.

\item We need a slightly different estimate for $| \nabla K |$ since $n > 2 \sigma + 3$.  
\end{itemize}

\vskip0.10in

\noindent {\bf Proof of Theorem \ref{Lar}.} The proof consists of seven steps.

\vskip0.10in

{\it Step 1.  Preliminaries.} Without loss of generality, we can assume that $0 < \varepsilon < 1$ and $k(0) = 1$. Since $\nabla k(0) = 0$, there exists a $C^1$ positive function $\widetilde{k} : \R^n \to \R$ such that $\widetilde{k} \equiv 1$ in a small neighborhood of the origin, $\widetilde{k} (x) = k(x)$ for $|x| \geq \varepsilon$ and $\| \widetilde{k} - k \|_{C^1 (\R^n)} < \varepsilon/2$. Replacing $k$ by $\widetilde{k}$, we can assume that $k \equiv 1$ in $B_\delta$ for some $0 < \delta < \varepsilon$.

Let
$$
h(r, \lambda) = c_{n, \sigma} \left( \frac{\lambda}{\lambda^2 + r^2} \right)^\frac{n - 2 \sigma}{2}
$$
and $\psi_\lambda (x)\ : = h (|x|, \lambda)$ for $\lambda>0$,  which are the well-known bubbles to the equation 
$$
(- \Delta)^\sigma \psi_\lambda = \psi_\lambda^\frac{n + 2 \sigma}{n - 2 \sigma} ~~~~~~ \textmd{in} ~ \R^n.
$$
After some calculations, one can find that there exist $\delta_1$ and $\delta_2$ satisfying
\begin{equation}\label{delta}
0 < \delta_2 < \frac{\delta_1}{2} < \frac{\delta}{4}
\end{equation}
and for any $|x| \leq \delta_2$ or $|x| \geq \delta$,
\begin{equation}\label{12-2}
\frac{1}{2} < \frac{ \psi_\lambda (x - x_1) }{ \psi_\lambda (x - x_2) } < 2 ~~~~~~ \textmd{when} ~ |x_1| = |x_2| = \delta_1 ~ \textmd{and} ~ 0 < \lambda \leq \delta_2.
\end{equation}

Recall that $k$ is bounded between two positive constants, we denote 
\begin{equation}\label{defab}
a = \frac{1}{2} \inf\limits_{\R^n} k > 0 ~~~~~~ \textmd{and} ~~~~~~ b = \sup\limits_{\R^n} k > 0.
\end{equation}
Let $i_0$ be the smallest integer greater than $2$ such that
\begin{equation}\label{defi0}
i_0^\frac{4 \sigma}{n - 2 \sigma} > \frac{ 2^\frac{3 n + 2 \sigma}{n - 2 \sigma} }{ (2 a)^\frac{n + 2 \sigma}{4 \sigma} }.
\end{equation}

Choose a sequence $\{ x_i \}_{i = 1}^\infty$ of distinct points in $\R^n$ and a sequence $\{ r_i \}_{i = 1}^\infty$ of positive numbers such that
\begin{equation}\label{defxi}
|x_1| = |x_2| = \cdots = | x_{i_0} | = \delta_1, ~~~~~~ \lim\limits_{i \to \infty} |x_i| = 0,
\end{equation}
\begin{equation}\label{defri}
r_1 = r_2 = \cdots = r_{i_0} = \frac{\delta_2}{2}, ~~~~~~ \lim\limits_{i \to \infty} r_i = 0,
\end{equation}
\begin{equation}\label{defBi1}
B_{4 r_i} (x_i) \subset B_{\delta_2} \setminus \{ 0 \} ~~~~~~ \textmd{for} ~ i > i_0
\end{equation}
and
\begin{equation}\label{defBi2}
\overline{B_{2 r_i} (x_i)} \cap \overline{B_{2 r_j} (x_j)} = \emptyset ~~~~~~ \textmd{for} ~ j > i > i_0.
\end{equation}
In addition, we require that the union of the line segments $\overline{x_1 x_2}$, $\overline{x_2 x_3}$, $\dots$, $\overline{ x_{i_0 - 1} x_{i_0} }$, $\overline{ x_{i_0} x_1 }$ be a regular $i_0$-gon. We will prescribe the side length of this polygon later. From \eqref{delta}, \eqref{defxi} and \eqref{defri} we know that
$$
\overline{B_{2 r_i} (x_i)} \subset B_{2 \delta_1} \setminus \overline{ B_{\delta_2} } ~~~~~~ \textmd{for} ~ 1 \leq i \leq i_0
$$
and hence by \eqref{defBi1},
\begin{equation}\label{defBi3}
\overline{B_{2 r_i} (x_i)} \cap \overline{B_{2 r_j} (x_j)} = \emptyset ~~~~~~ \textmd{for} ~ 1 \leq i \leq i_0 < j.
\end{equation}

Define three functions $f : [0, \infty) \times (0, \infty) \times (0, \infty) \to \R$ and $M, Z : (0, 1) \times (0, \infty) \to (0, \infty)$ by
$$
f(z_1, z_2, z_3) = z_2 (z_1 + z_3)^\frac{n + 2 \sigma}{n - 2 \sigma} - z_1^\frac{n + 2 \sigma}{n - 2 \sigma},
$$
\begin{equation}\label{defMZ}
M(z_2, z_3) = \frac{ z_2 z_3^\frac{n + 2 \sigma}{n - 2 \sigma} }{ \Big( 1 - z_2^\frac{n - 2 \sigma}{4 \sigma} \Big)^\frac{4 \sigma}{n - 2 \sigma} } ~~~~~~ \textmd{and} ~~~~~~ Z(z_2, z_3) = \frac{ z_3 z_2^\frac{n - 2 \sigma}{4 \sigma} }{ 1 - z_2^\frac{n - 2 \sigma}{4 \sigma} }.
\end{equation}
For each fixed $(z_2, z_3) \in (0, 1) \times (0, \infty)$, the function $f(\cdot, z_2, z_3) : [0, \infty) \to \R$ is strictly increasing on $[0, Z(z_2, z_3)]$ and is strictly decreasing on $[Z(z_2, z_3), \infty)$, and attains its maximum value $M(z_2, z_3)$ at $z_1 = Z(z_2, z_3)$.

Define $F : [0, \infty) \times (0, \infty) \times (0, \infty) \to (0, \infty)$ by
$$
F(z_1, z_2, z_3) =
\left\{
\aligned
& f(z_1, z_2, z_3) & \textmd{if} & ~ 0 < z_2 < 1 ~ \textmd{and} ~ 0 \leq z_1 \leq Z(z_2, z_3), \\
& M(z_2, z_3) & \textmd{if} & ~ 0 < z_2 < 1 ~ \textmd{and} ~ z_1 > Z(z_2, z_3), \\
& f(z_1, z_2, z_3) ~~~~~~ & \textmd{if} & ~ z_2 \geq 1.
\endaligned
\right.
$$
It is easy to see that $f$ and $F$ are $C^1$, $f \leq F$ and $F$ is non-decreasing in $z_1$, $z_2$ and $z_3$.

\vskip0.10in

{\it Step 2. Select the sequences $\{ x_i \}_{i = 1}^\infty$ and $\{ \lambda_i \}_{i = 1}^\infty$.} Let
$$
w(x) = (2 b)^{ - \frac{n}{2 \sigma} } h(|x|, 1) = \frac{ c_{n, \sigma} }{ (2 b)^\frac{n}{2 \sigma} } \left( \frac{1}{1 + |x|^2} \right)^\frac{n - 2 \sigma}{2} ~~~~~~ \textmd{for} ~ x \in \R^n.
$$
Then we have
\begin{equation}\label{Lsw}
(- \Delta)^\sigma w = (2 b)^\frac{2 n}{n - 2 \sigma} w^\frac{n + 2 \sigma}{n - 2 \sigma} ~~~~~~ \textmd{in} ~ \R^n 
\end{equation}
and by \eqref{defab},
\begin{equation}\label{supw0}
\sup\limits_{\R^n} w = w(0) = c_{n, \sigma} (2 b)^{ - \frac{n}{2 \sigma} } < c_{n, \sigma}.
\end{equation}

Choose a sequence $\{ \varepsilon_i \}_{i = 1}^\infty$ of positive numbers such that
\begin{equation}\label{defep}
\varepsilon_1 = \varepsilon_2 = \cdots = \varepsilon_{i_0} ~~~~~~ \textmd{and} ~~~~~~ \varepsilon_i \leq 2^{- i} ~~~ \textmd{for} ~ i \geq 1.
\end{equation}

Now we introduce four sequences of real numbers as follows. For $i \geq 1$, let
\begin{equation}\label{defki}
k_i \in \left( \frac{1}{2},1 \right) ~~~~~~ \textmd{with} ~~~~~~ k_1 = k_2 = \cdots = k_{i_0},
\end{equation}
\begin{equation}\label{defMi}
M_i = \frac{M( k_i, 2 w(0) )}{ ( 2 w(0) )^\frac{n + 2 \sigma}{n - 2 \sigma} } = \frac{k_i}{ \Big( 1 - k_i^\frac{n - 2 \sigma}{4 \sigma} \Big)^\frac{4 \sigma}{n - 2 \sigma} },
\end{equation}
\begin{equation}\label{defrh}
\rho_i = \sup \left\{ \rho > 0 ~ : ~ \mathcal{I}_{2 \sigma} (\chi_{B_{2 \rho} (x_i)}) \leq \frac{w}{ 2^{i + 1}( 2 w(0) )^\frac{n + 2 \sigma}{n - 2 \sigma} M_i } \right\}
\end{equation}
and
\begin{equation}\label{defla}
\lambda_i = \sup \Big\{ \lambda > 0 ~ : ~ \psi_\lambda (x - x_i) \leq \varepsilon_i a^\frac{n - 2 \sigma}{4 \sigma} w(x) ~ \textmd{for} ~ |x - x_i| \geq \rho_i \Big\}, 
\end{equation}
where $\mathcal{I}_{2 \sigma}$ is the Riesz potential operator of order $2 \sigma$ and $\chi_{B_{2 \rho} (x_i)}$ is the characteristic function of the ball $B_{2 \rho} (x_i)$. Then we have

\begin{lemma}\label{sim} For $i \geq 1$, 
$$
M_i \sim \frac{1}{ (1 - k_i)^\frac{4 \sigma}{n - 2 \sigma} }, ~~~ \rho_i^{2 \sigma} \sim \frac{1}{ 2^i M_i } ~~~ \textmd{and} ~~~ \lambda_i \sim \varepsilon_i^\frac{2}{n - 2 \sigma} \rho_i^2.
$$
\end{lemma}

\begin{proof} It is the same as that of \cite[Lemma 2.2]{DYla}.
\end{proof}

By Lemma \ref{sim}, after increasing the values of $k_i$ for certain values of $i$ while holding $\varepsilon_i$ fixed, we can assume for $i \geq 1$ that 
\begin{equation}\label{3seq1}
M_i > 9^i, ~~~~~~ \rho_i \in (0, r_i), ~~~~~~ \lambda_i \in (0, \delta_2),
\end{equation}
and
\begin{equation}\label{3seq2}
k_i^\frac{n + 2 \sigma}{4 \sigma} > \frac{ 1 + \big( \frac{1}{3} \big)^{n - 2 \sigma} }{ 1 + \frac{n + 2 \sigma}{n - 2 \sigma} \big( \frac{1}{3} \big)^{n - 2 \sigma} }, ~~~~~~ M_i^\beta > \max \left\{ \frac{1}{ \varepsilon_i^\frac{4 \sigma}{n - 2 \sigma} }, ~ 2^i \right\}, ~~~~~~ \lambda_i^\beta < \frac{ \varepsilon_i^\frac{2 \sigma}{n - 2 \sigma} }{2^i}, 
\end{equation}
where $\beta = \beta (n, \sigma) \in (0, 1/6)$ is a constant to be specified later.

Notice that for $1 \leq i \leq i_0$, $\rho_i$ and $\lambda_i$ do not change as $x_i$ moves on the sphere $|x| = \delta_1$. Therefore, we can require that the union of the line segments $\overline{x_1 x_2}$, $\overline{x_2 x_3}$, $\dots$, $\overline{ x_{i_0 - 1} x_{i_0} }$, $\overline{ x_{i_0} x_1 }$ be a regular $i_0$-gon with side length $4 \rho_1$. Thus,
\begin{equation}\label{distB}
\textmd{dist} ( B_{\rho_i} (x_i), B_{\rho_j} (x_j) ) \geq \rho_i + \rho_j ~~~~~~ \textmd{for} ~ 1 \leq i < j \leq i_0.
\end{equation}
By \eqref{defBi2}, \eqref{defBi3} and \eqref{3seq1}, the inequality \eqref{distB} also holds for $1 \leq i < j$.

For $i \geq 1$, define
$$
u_i (x) : = \psi_{\lambda_i} (x - x_i).
$$
Then one can check that
\begin{equation}\label{minBj}
\min\limits_{x \in B_{2 \rho_j} (x_j)} \frac{u_{j + 1} (x)}{u_{j - 1} (x)} > \left( \frac{1}{3} \right)^{n - 2 \sigma} ~~~~~~ \textmd{for} ~ 2 \leq j \leq i_0 - 1
\end{equation}
and a similar inequality holds when $j$ is $1$ or $i_0$.

Here we also give some inequalities which will be used later. By \eqref{3seq2}, Lemma \ref{sim}, the monotonicity of $Z$ and using the same arguments as in  \cite[(39)]{DYla},   we have for $1 \leq j \leq i_0$,
\begin{equation}\label{ZkMj}
\min\limits_{ x \in B_{2 \rho_j} (x_j) } Z \bigg( k_j^\frac{n + 2 \sigma}{4 \sigma}, \sum\limits_{i = 1, i \neq j}^{i_0} u_i (x) \bigg) \geq C M_j^\frac{(1 - \beta) (n - 2 \sigma)}{4 \sigma}.
\end{equation}
Thus, by increasing $k_1$ (recall that $k_1 = k_2 = \cdots = k_{i_0}$) we have
\begin{equation}\label{Zkw0}
\min\limits_{x \in B_{2 \rho_j} (x_j)} Z \bigg( k_j^\frac{n + 2 \sigma}{4 \sigma}, \sum\limits_{i = 1, i \neq j}^{i_0} u_i (x) \bigg) > w(0) ~~~~~~ \textmd{for} ~ 1 \leq j \leq i_0.
\end{equation}
By Lemma \ref{sim},
\begin{equation}\label{ZkjMj}
Z \Bigg( k_j^\frac{n + 2 \sigma}{4 \sigma}, \frac{1}{ 2 M_j^\frac{\beta (n - 2 \sigma) }{4 \sigma} } \Bigg) \sim \frac{1}{1 - k_j} \frac{1}{ M_j^\frac{\beta (n - 2 \sigma)}{4 \sigma} }\sim M_j^\frac{(1 - \beta) (n - 2 \sigma)}{4 \sigma} ~~~~~~ \textmd{for} ~ j \geq 1.
\end{equation}
Therefore, by increasing each term of the sequence $\{ k_j \}_{j = 1}^\infty$, we also have
$$
Z \Bigg( k_j^\frac{n + 2 \sigma}{4 \sigma}, \frac{1}{ 2 M_j^\frac{\beta (n - 2 \sigma)}{4 \sigma} } \Bigg) > w(0) ~~~~~~ \textmd{for} ~ j \geq 1.
$$
Then, by \eqref{defla}, \eqref{defab} and \eqref{defep} we have for $j \geq 1$ and $|x - x_j| \geq \rho_j$ that 
\begin{equation}\label{wxj}
u_j (x) \leq \varepsilon_j a^\frac{n - 2 \sigma}{4 \sigma} w(0) \leq w(0) < Z \Bigg( k_j^\frac{n + 2 \sigma}{4 \sigma}, \frac{1}{ 2 M_j^\frac{\beta (n - 2 \sigma) }{4 \sigma} } \Bigg).
\end{equation}

\vskip0.10in

{\it Step 3.  Estimate the sum of the bubbles $u_i$.} By \eqref{defla} and \eqref{defep} we have 
\begin{equation}\label{uil}
u_i \leq \varepsilon_i a^\frac{n - 2 \sigma}{4 \sigma} w ~~~~~~ \textmd{in} ~ \R^n \setminus B_{\rho_i} (x_i)
\end{equation}
and
\begin{equation}\label{uleqw}
\sum\limits_{i = 1}^\infty u_i \leq a^\frac{n - 2 \sigma}{4 \sigma} w ~~~~~~ \textmd{in} ~ \R^n - \bigcup\limits_{i = 1}^\infty B_{\rho_i} (x_i).
\end{equation}
By \eqref{defep} and \eqref{uil}, we know that $\sum_{i = 1}^\infty u_i \in C^\infty (\R^n \setminus \{ 0 \})$. Moreover, we claim that

\begin{lemma}\label{SR} $\sum_{i = 1}^\infty u_i \in L^\frac{n + 2 \sigma}{n - 2 \sigma} (\R^n) \cap L_\sigma (\R^n)$ and satisfies 
\begin{equation}\label{Lsui}
(- \Delta)^\sigma \bigg( \sum\limits_{i = 1}^\infty u_i \bigg) = \sum\limits_{i = 1}^\infty u_i^\frac{n + 2 \sigma}{n - 2 \sigma} ~~~~~~ \textmd{in} ~ \R^n \setminus \{0\}.
\end{equation}
\end{lemma}

\begin{proof} By \eqref{uil} and \eqref{uleqw}, we have
$$
\aligned
& \int_{\R^n} \bigg( \sum\limits_{i = 1}^\infty u_i (x) \bigg)^\frac{n + 2 \sigma}{n - 2 \sigma} dx \\
\leq & ~ \int_{\R^n - \bigcup\limits_{i = 1}^\infty B_{\rho_i} (x_i)} \Big( a^\frac{n - 2 \sigma}{4 \sigma} w \Big)^\frac{n + 2 \sigma}{n - 2 \sigma} dx + \sum\limits_{i = 1}^\infty \int_{B_{\rho_i} (x_i)} \Big( u_i (x) + a^\frac{n - 2 \sigma}{4 \sigma} w \Big)^\frac{n + 2 \sigma}{n - 2 \sigma} dx \\ 
\leq & ~ 2^\frac{n + 2 \sigma}{n - 2 \sigma} a^\frac{n + 2 \sigma}{4 \sigma} \int_{\R^n} w(x)^\frac{n + 2 \sigma}{n - 2 \sigma} dx + (2 c_{n, \sigma})^\frac{n + 2 \sigma}{n - 2 \sigma} \sum\limits_{i = 1}^\infty \int_{ B_{\rho_i} } \left( \frac{\lambda_i}{\lambda_i^2 + |x|^2} \right)^\frac{n + 2 \sigma}{2} dx \\
\leq & ~ C + C \sum\limits_{i = 1}^\infty \lambda_i^\frac{n - 2 \sigma}{2} \int_{B_{\rho_i/\lambda_i}} \left( \frac{1}{1 + |y|^2} \right)^\frac{n + 2 \sigma}{2} dy \\
 \leq & ~ C + C \sum\limits_{i = 1}^\infty \lambda_i^\frac{n - 2 \sigma}{2} < \infty.
\endaligned
$$
Therefore, $\sum_{i = 1}^\infty u_i \in L^\frac{n + 2 \sigma}{n - 2 \sigma} (\R^n)$. By H\"older's inequality, we have $\sum_{i = 1}^\infty u_i \in L_\sigma (\R^n)$.

Fix $x \in \R^n \setminus \{ 0 \}$. Since $\sum_{i = 1}^N u_i$ converges to $\sum_{i = 1}^\infty u_i$ in $C_{loc}^2 (\R^n \setminus \{ 0 \})$, we have
$$
\Bigg| 2 \sum_{i = 1}^N u_i (x) - \sum_{i = 1}^N u_i (x + y) - \sum_{i = 1}^N u_i (x - y) \Bigg| \leq C |y|^2 ~~~~~~ \textmd{for} ~ y \in B_{|x|/2},
$$
where $C$ is a positive constant when $N$ is sufficiently large, and
$$
\aligned
& |y|^{- n - 2 \sigma} \Bigg| 2 \sum_{i = 1}^N u_i (x) - \sum_{i = 1}^N u_i (x + y) - \sum_{i = 1}^N u_i (x - y) \Bigg| \\
\leq & ~ |y|^{- n - 2 \sigma} \bigg( 2 \sum_{i = 1}^\infty u_i (x) + \sum_{i = 1}^\infty u_i (x + y) + \sum_{i = 1}^\infty u_i (x - y) \bigg)
\endaligned
$$
for $y \in \R^n \setminus ( B_{|x|/2} \cup \{ x, - x \} )$. Since $\sum_{i = 1}^\infty u_i \in L_\sigma (\R^n)$, by Lebesgue's dominated convergence theorem, we get \eqref{Lsui}.
\end{proof}

Now, by increasing $k_i$ for each $i$, we can assume that
\begin{equation}\label{uixi}
u_i (x_i) = c_{n, \sigma} \lambda_i^{ - \frac{n - 2 \sigma}{2} } > i \varphi (|x_i|) ~~~~~~ \textmd{for} ~ i \geq 1
\end{equation} 
and $u_i + |\nabla u_i| < 2^{- i}$ in $\R^n \setminus B_{2 r_i} (x_i)$, $i \geq 1$. Thus by \eqref{defBi2} and \eqref{defBi3},
\begin{equation}\label{ujna}
u_i + |\nabla u_i| < 2^{- i} ~~~~~~ \textmd{in} ~ B_{2 r_j} (x_j) 
\end{equation}
when $i \neq j$ and either $i \geq 1$ and $j > i_0$ or $i > i_0$ and $1 \leq j \leq i_0$.  Again, by increasing $k_i$ for $i > i_0$, we can force $u_i$ and $M_i$ to satisfy
\begin{equation}\label{sinu1}
\sum\limits_{i = i_0 + 1}^\infty u_i (x) < \frac{1}{2} \min\limits_{1 \leq i \leq i_0} u_i (x) ~~~~~~ \textmd{for} ~ |x| \geq \delta_2,
\end{equation}
\begin{equation}\label{sumu12}
\sum\limits_{i = i_0 + 1, i \neq j}^\infty u_i (x) < \frac{u_1}{2} ~~~~~~ \textmd{in} ~ B_{2 r_j} (x_j), ~ j > i_0
\end{equation}
and
\begin{equation}\label{Minu1}
\frac{1}{ M_j^\frac{\beta (n - 2 \sigma)}{4 \sigma} } < \min\limits_{|x| \leq \delta} u_1 (x) ~~~~~~ \textmd{for} ~ j > i_0.
\end{equation}

It follows from \eqref{ujna} and \eqref{uil} that
\begin{equation}\label{uinC}
\sum\limits_{i = 1, i \neq j}^\infty u_i + u_i^\frac{n + 2 \sigma}{n - 2 \sigma} \leq C ~~~~~~ \textmd{in} ~ B_{\rho_j} (x_j), ~ j \geq 1.
\end{equation}
Similarly, by \eqref{ujna}, \eqref{uil}, Lemma \ref{sim} and \eqref{defi0},
$$
\aligned
& \sum\limits_{i = 1, i \neq j}^\infty |\nabla u_i| + u_i^\frac{4 \sigma}{n - 2 \sigma} |\nabla u_i| \leq \sum\limits_{i = 1, i \neq j}^{i_0} |\nabla u_i| + u_i^\frac{4 \sigma}{n - 2 \sigma} |\nabla u_i| + C \\
& ~~~ \leq C \sum\limits_{i = 1, i \neq j}^{i_0} u_i \frac{1}{\rho_j} + u_i^\frac{n + 2 \sigma}{n - 2 \sigma} \frac{1}{\rho_j} + C \\
& ~~~ \leq C 2^\frac{j}{2 \sigma} M_j^\frac{1}{2 \sigma} \leq C 2^\frac{i_0}{2 \sigma} M_j^\frac{1}{2 \sigma} \leq C M_j^\frac{1}{2 \sigma} ~~~~~~ \textmd{in} ~ B_{\rho_j} (x_j), ~ 1 \leq j \leq i_0,  
\endaligned
$$
and by \eqref{ujna} and \eqref{uil},
$$
\sum\limits_{i = 1, i \neq j}^\infty |\nabla u_i| + u_i^\frac{4 \sigma}{n - 2 \sigma} |\nabla u_i| \leq C ~~~~~~ \textmd{in} ~ B_{\rho_j} (x_j), ~ j > i_0.
$$
Thus, we get
\begin{equation}\label{laui}
\sum\limits_{i = 1, i \neq j}^\infty |\nabla u_i| + u_i^\frac{4 \sigma}{n - 2 \sigma} |\nabla u_i| \leq C M_j^\frac{1}{2 \sigma} ~~~~~~ \textmd{in} ~ B_{\rho_j} (x_j), ~ j \geq 1. 
\end{equation}

\vskip0.10in

{\it Step 4. Construct the correction function $u_0$.} Since $n > 2 \sigma + 3$, by Lemma \ref{sim} and \eqref{3seq1} we have
\begin{equation}\label{1kir}
\frac{1 - k_i}{\rho_i} \sim \frac{ 2^\frac{i}{2 \sigma} M_i^\frac{1}{2 \sigma} }{ M_i^\frac{n - 2 \sigma}{4 \sigma} } \sim \frac{ 2^\frac{i}{2 \sigma} }{ M_i^\frac{n - 2 \sigma - 2}{4 \sigma} } \leq \frac{ 2^\frac{i}{2 \sigma} }{ M_i^\frac{1}{4 \sigma} } \leq \left( \frac{2}{3} \right)^\frac{i}{2 \sigma} \to 0 ~~~~~~ \textmd{as} ~ i \to \infty.
\end{equation}
Let $\eta : [0, \infty) \to [0, 1]$ be a $C^\infty$ cut-off function satisfying $\eta (t) = 1$ for $0 \leq t \leq 1$ and $\eta (t) = 0$ for $t \geq 3/2$. Define
\begin{equation}\label{defka}
\kappa (x) = k(x) + \sum\limits_{i = 1}^\infty ( k_i - k(x) ) \eta_i (x) ~~~~~~ \textmd{for} ~ x \in \R^n, 
\end{equation}
where $\eta_i (x) : = \eta (|x - x_i|/\rho_i)$. Since $\{ \eta_i \}_{i = 1}^\infty$ have disjoint supports contained in $B_{2 \delta_1} \setminus \{ 0 \}$, $\kappa$ is well-defined. Recall that $k = 1$ in $B_{\delta}$, we obtain $\kappa (0) = k(0) = 1$, $\kappa \leq k$ in $\R^n$ and $\kappa (x) = k(x)$ for $|x| \geq 2 \delta_1$. By \eqref{defab} and \eqref{defki} we have
\begin{equation}\label{infka}
\inf\limits_{\R^n} \kappa \geq a.
\end{equation}
Since $k \equiv 1$ in $B_\delta$, then 
\begin{equation}\label{blaka}
\nabla \kappa (x) = \sum\limits_{i = 1}^\infty \frac{k_i - 1}{\rho_i} \eta' \left( \frac{|x - x_i|}{\rho_i} \right) \frac{x - x_i}{|x - x_i|} ~~~~~~ \textmd{for} ~ x \in B_\delta,
\end{equation}
it follows from \eqref{1kir} that $\kappa \in C^1 (\R^n)$ and $\nabla \kappa (0) = 0$.

By \eqref{defrh},
\begin{equation}\label{IsMw2}
0 < \mathcal{I}_{2 \sigma} \overline{M} < \frac{w}{2} ~~~~~~ \textmd{in} ~ \R^n, 
\end{equation}
where
$$
\overline{M} (x) : =
\begin{cases}
( 2 w(0) )^\frac{n + 2 \sigma}{n - 2 \sigma} M_i ~~~~~~ & \textmd{in} ~ B_{\rho_i} (x_i), ~ i \geq 1, \\
0 ~~~~~~ & \textmd{in} ~ \R^n - \bigcup\limits_{i = 1}^\infty B_{2 \rho_i} (x_i), \\
( 2 w(0) )^\frac{n + 2 \sigma}{n - 2 \sigma} M_i \left( 2 - \frac{|x - x_i|}{\rho_i} \right) ~~~~~~ & \textmd{in} ~ B_{2 \rho_i} (x_i) \setminus B_{\rho_i} (x_i), ~ i \geq 1. 
\end{cases}
$$
Since $\overline{M}$ is locally Lipschitz continuous in $\R^n \setminus \{ 0 \}$ and $\overline{M} \in L^\frac{n}{2 \sigma} (\R^n)$, we have 
$$
\overline{v} : = \frac{w}{2 b} +\mathcal{I}_{2 \sigma} \overline{M} \in C_{loc}^{2 \sigma + \gamma}(\R^n \setminus \{ 0 \})
$$
for any $\gamma \in (0, 1)$. It follows from \eqref{supw0} and \eqref{IsMw2} that
\begin{equation}\label{w2bvw}
\frac{w}{2b} < \overline{v} < w < c_{n, \sigma} ~~~~~~ \textmd{in} ~ \R^n.
\end{equation}
Hence, we get $\overline{v} \in L_\sigma (\R^n)$. By \eqref{Lsw}, 
\begin{equation}\label{defov}
(- \Delta)^\sigma \overline{v} = (2 b)^\frac{n + 2 \sigma}{n - 2 \sigma} w^\frac{n + 2 \sigma}{n - 2 \sigma} + \overline{M} ~~~~~~ \textmd{in} ~ \R^n \setminus \{ 0 \}.
\end{equation}

Define $\underline{H} : \R^n \times [0, \infty) \to \R$ by
\begin{equation}\label{underH}
\underline{H} (x, v) = \kappa (x) \bigg( v + \sum\limits_{i = 1}^\infty u_i (x) \bigg)^\frac{n + 2 \sigma}{n - 2 \sigma} - \sum\limits_{i = 1}^\infty u_i (x)^\frac{n + 2 \sigma}{n - 2 \sigma}. 
\end{equation}
Then we have
$$
\underline{H} (x, v) = f( \widetilde{u} (x), \kappa (x), p(x, v) ), 
$$
where
$$
\widetilde{u} (x) : = \bigg( \sum\limits_{i = 1}^\infty u_i (x)^\frac{n + 2 \sigma}{n - 2 \sigma} \bigg)^\frac{n - 2 \sigma}{n + 2 \sigma} ~~~~~~ \textmd{and} ~~~~~~ p(x, v) : =  v + \sum\limits_{i = 1}^\infty u_i (x) - \widetilde{u} (x).
$$

Define $H : \R^n \times [0, \infty) \to (0, \infty)$ by
\begin{equation}\label{defH}
H(x, v) = F( \widetilde{u} (x), \kappa (x),  p(x, v) ).
\end{equation}
Then
\begin{equation}\label{HleqM}
H(x,v) \leq M( \kappa (x), p(x, v) ) ~~~~~~ \textmd{when} ~ \kappa (x) < 1.
\end{equation}
Moreover, by the definition of $F$ we have that $H(x, v) = \underline{H} (x, v)$ if and only if either $\kappa (x) < 1$ and $\widetilde{u} (x) \leq Z(\kappa (x), p(x,v))$ or $\kappa (x) \geq 1$.

For $x \in \R^n - \bigcup\limits_{i = 1}^\infty B_{\rho_i} (x_i)$ and $\kappa (x) < 1$, we have
$$
\aligned
\widetilde{u} (x) & \leq \sum\limits_{i = 1}^\infty u_i (x) \leq a^\frac{n - 2 \sigma}{4 \sigma} w(x) ~~~~~~ \textmd{by} ~ \eqref{uleqw} \\
& \leq \frac{ w(x) \kappa (x)^\frac{n - 2 \sigma}{4 \sigma} }{ 1 - \kappa (x)^\frac{n - 2 \sigma}{4 \sigma} } ~~~~~~ \textmd{by} ~ \eqref{infka} \\
& \leq \frac{ p( x, w(x) ) \kappa (x)^\frac{n - 2 \sigma}{4 \sigma} }{ 1 - \kappa (x)^\frac{n - 2 \sigma}{4 \sigma} } \\
& = Z( \kappa (x), p( x, w(x) ) ) ~~~~~~ \textmd{by} ~ \eqref{defMZ}.
\endaligned
$$
Hence
$$
H( x, w(x) ) = \underline{H} ( x, w(x) ) ~~~~~~ \textmd{for} ~ x \in \R^n - \bigcup\limits_{i = 1}^\infty B_{\rho_i} (x_i).
$$
Thus, for $x \in (\R^n \setminus \{ 0 \}) - \bigcup\limits_{i = 1}^\infty B_{\rho_i} (x_i)$ and $0 \leq v \leq w(x)$  we have
\begin{equation}\label{Hxv}
\aligned
H(x, v) \leq H( x, w(x) ) = \underline{H} ( x, w(x) ) & \leq \kappa (x) \bigg( w(x) + \sum\limits_{i = 1}^\infty u_i (x) \bigg)^\frac{n + 2 \sigma}{n - 2 \sigma} \\
& \leq b ( 2 w(x) )^\frac{n + 2 \sigma}{n - 2 \sigma} \leq (- \Delta)^\sigma \overline{v} (x)
\endaligned
\end{equation}
by \eqref{uleqw}, \eqref{defab} and \eqref{defov}.

For $x \in B_{\rho_i} (x_i)$ and $i \geq 1$ we have $\kappa (x) \equiv  k_i < 1$. Hence, from \eqref{HleqM} we obtain for $x \in B_{\rho_i} (x_i)$ and $0 \leq v \leq w(x)$ that 
\begin{equation}\label{Hxv2}
\aligned
H(x, v) & \leq M ( k_i, p(x, v) ) = \frac{ k_i p(x, v)^\frac{n + 2 \sigma}{n - 2 \sigma} }{ \Big( 1 - k_i^\frac{n - 2 \sigma}{4 \sigma} \Big)^\frac{4 \sigma}{n - 2 \sigma} } \\
& \leq M_i \bigg( v + \sum\limits_{j = 1}^\infty u_j (x) - \widetilde{u} (x) \bigg)^\frac{n + 2 \sigma}{n - 2 \sigma} ~~~~~~ \textmd{by} ~ \eqref{defMi} \\
& \leq M_i \bigg( v + \sum\limits_{j = 1, j \neq i}^\infty u_j (x) \bigg)^\frac{n + 2 \sigma}{n - 2 \sigma} \\
& \leq M_i ( 2 w(x) )^\frac{n + 2 \sigma}{n - 2 \sigma} ~~~~~~ \textmd{by} ~ \eqref{uil} \\
& \leq M_i ( 2 w(0) )^\frac{n + 2 \sigma}{n - 2 \sigma} = \overline{M} (x) \leq (- \Delta)^\sigma \overline{v} (x) ~~~~~~ \textmd{by} ~ \eqref{defov}. 
\endaligned
\end{equation}
This together with \eqref{Hxv} implies that
\begin{equation}\label{sups}
H(x, v) \leq (- \Delta)^\sigma \overline{v} (x) ~~~~~~ \textmd{for} ~ x \in \R^n \setminus \{ 0 \} ~ \textmd{and} ~ 0 \leq v \leq w(x).
\end{equation}

Hence, by the non-negativity of $H$, \eqref{w2bvw} and \eqref{sups} we can use  $\underline{v} \equiv 0$ and $\overline{v}$ as sub- and super-solutions of the problem 
$$
(- \Delta)^\sigma u = H(x, u) ~~~~~~ \textmd{in} ~ \R^n \setminus \{ 0 \}.
$$
Now, applying the method of sub- and super-solutions, we can get the correction function $u_0$ by the following two lemmas.

\begin{lemma}\label{subm} There exists a $C^\sigma (\R^n) \cap C^{2 \sigma + \gamma} (B_2 \setminus \overline{B_1})$ solution $v$ to 
\begin{equation}\label{v-B2B1}
\left\{
\aligned
& (- \Delta)^\sigma v = H(x, v) ~~~~~~ & \textmd{in} & ~ B_2 \setminus \overline{B_1}, \\
& v = 0 & \textmd{on} & ~ B_2^c \cup \overline{B_1}
\endaligned
\right.
\end{equation}
for every $\gamma \in (0, 1)$.
\end{lemma}

\begin{proof} Let $\{ y_i \}_{i = 1}^\infty$ be the sequence of functions defined in $\R^n$ by the iteration scheme
$$
\left\{
\aligned
& (- \Delta)^\sigma y_{i + 1} = H(x, y_i) ~~~~~~ & \textmd{in} & ~ B_2 \setminus \overline{B_1}, \\
& y_{i+1} = 0 & \textmd{on} & ~ B_2^c \cup \overline{B_1}.
\endaligned
\right.
$$
where $y_0 : = 0$. Then we have 
$$
\left\{
\aligned
& (- \Delta)^\sigma (y_1 - y_0) = H(x, y_0) - (- \Delta)^\sigma y_0 \geq 0 ~~~~~~ & \textmd{in} & ~ B_2 \setminus \overline{B_1}, \\
& y_1 - y_0 = 0 & \textmd{on} & ~ B_2^c \cup \overline{B_1}.
\endaligned
\right.
$$
Hence, $y_1 \geq y_0 = 0$ in $\R^n$ by the maximum principle. If $y_i \geq y_{i - 1}$ for $i \geq 1$, then $(- \Delta)^\sigma (y_{i + 1} - y_i) \geq 0$ in $B_2 \setminus \overline{B_1}$ and $y_{i + 1} - y_i = 0$ on $B_2^c \cup \overline{B_1}$, which implies that $y_{i + 1} \geq y_i$ in $\R^n$. By induction, $\{ y_i \}_{i = 1}^\infty$ is a non-decreasing sequence of functions in $\R^n$. A similar application of the maximum principle for $y_i$ and $\overline{v}$ yields that $y_i \leq \overline{v}$ in $B_2 \setminus \overline{B_1}$ for all $i \geq 1$. Consequently $y_i$ converges point-wise to a limit $v$ in $\R^n$. Moreover, for every $i \geq 1$, by \eqref{w2bvw},
$$
0 \leq H(x, y_i) \leq H(x, \overline{v}) \leq H(x, c_{n, \sigma}) < \infty ~~~~~~ \textmd{in} ~ B_2 \setminus \overline{B_1}.
$$
By \cite[Proposition 1.1]{Ro-14}, we have $y_i \in C^\sigma (\R^n)$ and
\begin{equation}\label{yi-Csi}
\| y_i \|_{C^\sigma (\R^n)} \leq C \sup\limits_{x \in \overline{B_2} \setminus B_1} H(x, c_{n, \sigma}) < \infty ~~~~~~ \textmd{for all} ~ i \geq 1
\end{equation}
where $C$ is a constant depending only on $n$ and $\sigma$. Since $H \in C^1(\overline{B_2} \setminus B_1 \times [0, c_{n, \sigma}])$, by \eqref{yi-Csi} and \cite[Proposition 1.4]{Ro-14}, $y_i \in C^{2 \sigma + \gamma} (B_2 \setminus \overline{B_1})$ for any $\gamma \in (0, 1)$. Moreover, for any compact set $K \subset B_2 \setminus \overline{B_1}$, there exists $C = C(K)$ such that $\| y_i \|_{C^{2 \sigma + \gamma} (K)} \leq C$ for all $i \geq 1$. Thus,
$$
y_i \to v ~~~~~~ \textmd{in} ~ C^{2 \sigma + \gamma}_{loc} (B_2 \setminus \overline{B_1}).
$$
It implies that $v$ satisfies \eqref{v-B2B1}. Recall that $v \leq \overline{v}$ in $B_2 \setminus \overline{B_1}$, by \cite[Corollary 1.6]{Ro-14}, $v \in C^\sigma (\R^n) \cap C^{2 \sigma + \gamma} (B_2 \setminus \overline{B_1})$ for every $\gamma \in (0, 1)$.
\end{proof}

Now we can construct the desired function $u_0$.

\begin{lemma}\label{supm} There exists a $C^{2 \sigma + \gamma} (\R^n \setminus \{ 0 \})$ solution $u_0$ to  
\begin{equation}\label{defu0}
\left\{
\aligned
& (- \Delta)^\sigma u_0 = H(x, u_0) \\
& 0 \leq u_0 \leq \overline{v} \leq w
\endaligned
\right. ~~~~~~ \textmd{in} ~ \R^n \setminus \{ 0 \}.
\end{equation}
\end{lemma}

\begin{proof} For each positive integer $i \geq 2$, we consider about the following problem  
\begin{equation}\label{loceq}
\left\{
\aligned
& (- \Delta)^\sigma v = H(x, v) ~~~~~~ & \textmd{in} & ~ B_i \setminus \overline{ B_{1/i} }, \\
& v = 0 & \textmd{on} & ~ B_i^c \cup \overline{ B_{1/i} }.
\endaligned
\right.
\end{equation}
Using the same argument in Lemma \ref{subm}, \eqref{loceq} has a non-negative solution $v_i$ such that $0 \leq v_i \leq \overline{v}$ in $\R^n$ and $v_i \in C^\sigma (\R^n) \cap C_{loc}^{2 \sigma + \gamma} (B_i \setminus \overline{ B_{1/i} })$. Moreover, $\{ v_i \}_{i = 2}^\infty$ has a uniform upper bound $w$. It follows from \cite[Proposition 1.4]{Ro-14} and the standard diagonal argument that, after passing to a subsequence, 
$$
v_i \to u_0 ~~~~~~ \textmd{in} ~ C^{2 \sigma + \gamma}_{loc} (\R^n \setminus \{ 0 \})
$$
for any $\gamma \in (0, 1)$. Since $0 \leq u_0 \leq \overline{v} \leq w$ in $\R^n \setminus \{ 0 \}$, $u_0 \in L_\sigma (\R^n)$. Fixing $x \in \R^n \setminus \{ 0 \}$, we have (pick $\gamma = 1 - \sigma$)
$$
| 2 v_i (x) - v_i (x + y) - v_i (x - y) | \leq C |y|^{1 - \sigma} ~~~~~~ \textmd{for} ~ y \in B_{|x|/2}
$$
where $C$ does not depend on $i$, and
$$
|y|^{- n - 2 \sigma} | 2 v_i (x) - v_i (x + y) - v_i (x - y)| \leq |y|^{- n - 2 \sigma} ( 2 \overline{v} (x) + \overline{v} (x + y) + \overline{v} (x - y) )
$$
for $y \in \R^n \setminus B_{|x|/2}$. Since $\overline{v} \in L_\sigma(\R^n)$, by Lebesgue's dominated convergence theorem, we get \eqref{defu0}.
\end{proof}

\begin{lemma}\label{disu0} Let $u_0 \in C^{2 \sigma + \gamma} (\R^n \setminus \{ 0 \}) \cap L^\infty (\R^n)$ be a solution of \eqref{defu0}.   Then $u_0$ is a distributional solution in $\R^n$, i.e.,
$$
\int_{\R^n} u_0 (- \Delta)^\sigma \phi dx = \int_{\R^n} H(x, u_0) \phi dx ~~~~~~ \textmd{for any} ~ \phi \in C_c^\infty (\R^n). 
$$
\end{lemma}

\begin{proof} Let $\eta : \R^n \to [0, 1]$ be a $C^\infty$ cut-off function satisfying $\eta = 1$ in $\R^n \setminus B_2$ and  $\eta = 0$ in $B_1$.  Let $\eta_\varepsilon (x) : = \eta (\varepsilon^{- 1} x)$ for $\varepsilon \in (0, 1)$. For any $\phi \in C_c^\infty (\R^n)$, we use $\phi \eta_\varepsilon$ as a  test function.  From Lemma \ref{supm}, we know that 
\begin{equation}\label{ibpu0}
\int_{\R^n} (\phi \eta_\varepsilon) H(x, u_0) dx = \int_{\R^n} (\phi \eta_\varepsilon) (- \Delta)^\sigma u_0 dx = \int_{\R^n} u_0 (- \Delta)^\sigma (\phi \eta_\varepsilon) dx. 
\end{equation}
Denote 
\begin{equation}\label{Dtfx}
I_\varepsilon (x) : = (- \Delta)^\sigma (\phi \eta_\varepsilon) (x) - \eta _\varepsilon (x) (- \Delta)^\sigma \phi (x).
\end{equation}
When $|x| \leq 4 \varepsilon$,
$$
\aligned
|I_\varepsilon (x)| & \leq C + C \int_{|z| \leq 2 \varepsilon} \frac{ | 2 (\phi \eta_\varepsilon) (x) - (\phi \eta_\varepsilon) (x + z) - (\phi \eta_\varepsilon) (x - z) | }{ |z|^{n + 2 \sigma} } dz + C \int_{|z| \geq 2 \varepsilon} \frac{1}{ |z|^{n + 2 \sigma} } dz \\
& \leq C + C \varepsilon^{- 2} \int_{|z| \leq 2 \varepsilon} \frac{|z|^2}{ |z|^{n + 2 \sigma} } dz + C \varepsilon^{- 2 \sigma} \leq C \varepsilon^{- 2 \sigma}.
\endaligned
$$
When $|x| > 4 \varepsilon$, we have $\eta_\varepsilon (x) = 1$, then
$$
\aligned
|I_\varepsilon (x)| & = C \bigg| \int_{\R^n} \frac{ \phi (y) ( \eta_\varepsilon (x) - \eta_\varepsilon (y) ) }{ |x - y|^{n + 2 \sigma} } dy \bigg| \leq C \int_{|y| \leq 2 \varepsilon} \frac{1}{ |x - y|^{n + 2 \sigma} } dy \\
& \leq C \int_{|y| \leq 2 \varepsilon} \frac{1}{ |x|^{n + 2 \sigma} } dy \leq C \varepsilon^n |x|^{- n - 2 \sigma}.
\endaligned
$$
Thus
$$
\aligned
\bigg| \int_{\R^n} u_0 I_\varepsilon (x) dx \bigg| & \leq C \varepsilon^{- 2 \sigma} \int_{|x| \leq 4 \varepsilon} u_0 dx + C \varepsilon^n \int_{|x| > 4 \varepsilon} u_0 |x|^{- n - 2 \sigma} dx \\
& \leq C \varepsilon^{n - 2 \sigma} \to 0
\endaligned
$$
as $\varepsilon \to 0$. Since $u_0 \in L^\infty (\R^n)$ and $H(x, u_0) \in L_{loc}^1 (\R^n)$, by sending $\varepsilon \to 0$ in \eqref{ibpu0} we obtain  that $u_0$ is a distributional solution in $\R^n$.  
\end{proof}

{\it Step 5. Define the functions $u$ and $K$.} Define $\overline{H} : \R^n \times [0, \infty) \to (0, \infty)$ by $\overline{H} (x, v) = F( \widetilde{u} (x), k(x), p(x, v) )$. Then $\underline{H} \leq H \leq \overline{H}$ since $\kappa \leq k$. In particular,
\begin{equation}\label{uHH}
\underline{H} ( x, u_0 (x) ) \leq H( x, u_0(x) ) \leq \overline{H} ( x, u_0(x) ) ~~~~~~ \textmd{for} ~ x \in \R^n \setminus \{ 0 \}.
\end{equation}

Similar to the arguments in \cite[Step 4]{DYla},  we have for $x \in \R^n$ and $v \geq 0$,  
$$
\aligned
\overline{H} (x, v) & = f( \widetilde{u} (x), k(x), p(x, v) ) \\
& = k(x) \bigg( v + \sum\limits_{i = 1}^\infty u_i (x) \bigg)^\frac{n + 2 \sigma}{n - 2 \sigma} - \sum\limits_{i = 1}^\infty u_i (x)^\frac{n + 2 \sigma}{n - 2 \sigma}.
\endaligned
$$
Moreover, define 
\begin{equation}\label{Solu=0h87}
u : = u_0 + \sum\limits_{i = 1}^\infty u_i.  
\end{equation} 
Then $u \in L_\sigma (\R^n) \cap C^{2 \sigma + \gamma} (\R^n \setminus \{ 0 \})$ is a positive solution of   
\begin{equation}\label{defu}
\kappa (x) u^\frac{n + 2 \sigma}{n - 2 \sigma} \leq (- \Delta)^\sigma u \leq k(x)  u^\frac{n + 2 \sigma}{n - 2 \sigma} ~~~~~~ \textmd{in} ~ \R^n \setminus \{0\}.
\end{equation}
It follows from \eqref{uixi}, \eqref{sinu1} and \eqref{defu0} that $u$ satisfies \eqref{u-phi} and \eqref{u-inf}.

Define $K : \R^n \to (0, \infty)$ by
\begin{equation}\label{defK}
K(x) := \frac{(- \Delta)^\sigma u(x)}{ u(x)^\frac{n + 2 \sigma}{n - 2 \sigma} } ~~~~~~ \textmd{for} ~ x \in \R^n \setminus \{ 0 \}
\end{equation}
and $K(0) = 1$. Then
\begin{equation}\label{KxH}
K(x) = \frac{ H( x, u_0(x) ) + \sum_{i = 1}^\infty u_i (x)^\frac{n + 2 \sigma}{n - 2 \sigma} }{ \big( u_0 (x) + \sum_{i = 1}^\infty u_i (x) \big)^\frac{n + 2 \sigma}{n - 2 \sigma} } ~~~~~~ \textmd{for} ~ x \in \R^n \setminus \{ 0 \}
\end{equation}
and hence $K \in C^1 (\R^n \setminus \{ 0 \})$. It follows from \eqref{defu} and \eqref{defK} that
\begin{equation}\label{kKk}
\kappa (x) \leq K(x) \leq k(x) ~~~~~~ \textmd{for} ~ x \in \R^n \setminus \{ 0 \}.
\end{equation}
Recall that $\kappa$, $k \in C^1 (\R^n)$ and $\kappa (0) = K(0) = k(0) = 1$, we get $K \in C (\R^n)$,
\begin{equation}\label{nablakKk}
\nabla \kappa (0) = \nabla K(0) = \nabla k(0) = 0
\end{equation}
and
\begin{equation}\label{keqKeqk}
\kappa (x) = K(x) = k(x) ~~~~~~ \textmd{for} ~ |x| \geq 2 \delta_1.
\end{equation}

\vskip0.10in

{\it Step 6. Show that $K \in C^1 (\R^n)$.}  We only need to show that
\begin{equation}\label{limnablaK}
\lim\limits_{|x| \to 0^+} \nabla K(x) = 0.
\end{equation}
Let $S = \{ x \in \R^n \setminus \{ 0 \} ~ : ~ \underline{H} ( x, u_0 (x) ) < H( x,u_0 (x) ) \}$.

It follows from \eqref{underH} and \eqref{KxH} that
\begin{equation}\label{defS}
S = \{ x \in \R^n \setminus \{ 0 \} ~ : ~ \kappa (x) < K(x) \}. 
\end{equation}
By \eqref{kKk}, \eqref{nablakKk} and \eqref{defS} we obtain 
\begin{equation}\label{nkaeqnK}
\nabla \kappa (x) = \nabla K(x) ~~~~~~ \textmd{for} ~ x \in \R^n \setminus S
\end{equation}
and thus \eqref{limnablaK} holds for $x \in (\R^n \setminus \{ 0 \}) - S$. Next we show that \eqref{limnablaK} holds for $x \in S$. It follows from \eqref{underH} and \eqref{defH} that
\begin{equation}\label{SHMUZ}
\left\{
\aligned
& H( x, u_0(x) ) = M( \kappa (x), p_0 (x) ) \\
& \widetilde{u} (x) > Z( \kappa (x), p_0 (x) )
\endaligned
\right. ~~~~~~ \textmd{for} ~ x \in S, 
\end{equation}
where $p_0 (x) := p( x, u_0 (x) )$. Since $\kappa \geq k_j$ in $B_{2 \rho_j} (x_j)$, by \eqref{defMZ} we have 
\begin{equation}\label{UxZkjMj}
\widetilde{u} (x) > Z( k_j, p_0 (x) ) = M_j^\frac{n - 2 \sigma}{4 \sigma} p_0 (x) ~~~~~~ \textmd{for} ~ x \in S \cap B_{2 \rho_j} (x_j), ~ j \geq 1.
\end{equation}
For $x' \in (\R^n \setminus \{ 0 \}) - \bigcup\limits_{i = 1}^\infty B_{2 \rho_i} (x_i)$, $\kappa (x') = k (x')$. Hence, by \eqref{kKk} and \eqref{defS} we know that $x' \not\in S$. Consequently,
\begin{equation}\label{SsubBi}
S \subset \bigcup\limits_{i = 1}^\infty B_{2 \rho_i} (x_i).
\end{equation}

It follows from the same proofs as in \cite[(86) and (87)]{DYla} that
\begin{equation}\label{SB2SB}
S \cap B_{2 \rho_j} (x_j) = S \cap B_{\rho_j} (x_j) ~~~~~~ \textmd{for}~ j \geq 1
\end{equation} 
and
\begin{equation}\label{ujMjj1}
u_j \geq C M_j^\frac{ (1 - \beta) (n - 2 \sigma) }{4 \sigma} ~~~~~~ \textmd{in} ~ S \cap B_{2 \rho_j} (x_j) , ~ j \geq 1.
\end{equation}

Recall \eqref{KxH} and \eqref{SHMUZ}, we have 
$$
K(x) = \frac{ M_j p_0 (x)^\frac{n + 2 \sigma}{n - 2 \sigma} + \widetilde{u} (x)^\frac{n + 2 \sigma}{n - 2 \sigma} }{ \big( p_0 (x) + \widetilde{u} (x) \big)^\frac{n + 2 \sigma}{n - 2 \sigma} } = \frac{ M_j \Big( \frac{p_0 (x)}{\widetilde{u} (x)} \Big)^\frac{n + 2 \sigma}{n - 2 \sigma} + 1 }{ \Big( \frac{p_0 (x)}{\widetilde{u} (x)} + 1 \Big)^\frac{n + 2 \sigma}{n - 2 \sigma} } ~~~~~~ \textmd{for} ~ x \in S \cap B_{\rho_j} (x_j) , ~ j \geq 1.
$$
Thus
$$
\nabla K = \frac{n + 2 \sigma}{n - 2 \sigma} \Bigg( \frac{ M_j \big( \frac{p_0}{ \widetilde{u} } \big)^\frac{4 \sigma}{n - 2 \sigma} - 1 }{ \big( \frac{p_0}{\widetilde{u}} + 1 \big)^\frac{2 n}{n - 2 \sigma}} \Bigg) \bigg( \nabla \frac{p_0}{\widetilde{u} (x)} \bigg) ~~~~~~ \textmd{in} ~ S \cap B_{\rho_j} (x_j), ~ j \geq 1.
$$
and hence, by \eqref{UxZkjMj},
\begin{equation}\label{nablaK3t}
\aligned
|\nabla K| & \leq \frac{n + 2 \sigma}{n - 2 \sigma} \Big| \nabla \frac{p_0}{ \widetilde{u} } \Big| \\
& \leq \frac{n + 2 \sigma}{n - 2 \sigma} \Bigg[ \Big| \nabla \frac{u_0}{ \widetilde{u} } \Big| + \bigg| \nabla \frac{\sum_{i = 1, i \neq j}^\infty u_i}{ \widetilde{u} } \bigg| + \Big| \nabla\frac{u_j}{ \widetilde{u} } \Big| \Bigg] ~~~~~~ \textmd{in} ~ S \cap B_{\rho_j} (x_j), ~ j \geq 1.
\endaligned
\end{equation}

Using the same algebraic calculations as in \cite[Step 5]{DYla},  we obtain for $j \geq 1$ and  $x \in S \cap B_{\rho_j} (x_j)$, 
\begin{equation}\label{balabala} 
|\nabla K| \leq  C \Bigg( \frac{ |\nabla u_0| }{ M_j^\frac{(1 - \beta) (n - 2 \sigma)}{4 \sigma} } + \frac{| \nabla u_j |}{u_j^2} + \frac{ M_j^\frac{1}{2 \sigma} }{ M_j^\frac{(1 - \beta) (n - 2 \sigma)}{4 \sigma} } \Bigg).
\end{equation}

We now estimate the first term in \eqref{balabala}. By Lemma \ref{disu0} and \cite[Proposition 2.22]{Si-07},  there exists a function $h \in C(B_2) \cap L_\sigma (\R^n) $ satisfying  $(- \Delta)^\sigma h = 0$ in $B_2$ such that   
$$
u_0 (x) = r_{n, \sigma} \int_{B_4} \frac{H( y, u_0 (y) )}{ |x - y|^{n - 2 \sigma} } dy + h(x) ~~~~~~ \textmd{for} ~ 0 < |x| \leq 2.
$$
By \eqref{Hxv}, \eqref{Hxv2} and \eqref{defu0}, 
$$
H( x, u_0 (x) ) \leq 
\begin{cases}
( 2 w(0) )^\frac{n + 2 \sigma}{n - 2 \sigma} M_j ~~~~~~  & \textmd{in} ~ B_{\rho_j} (x_j), ~ j \geq 1, \\
( 2 w(0) )^\frac{n + 2 \sigma}{n - 2 \sigma} b & \textmd{in} ~ (\R^n \setminus \{0\}) - \bigcup\limits_{i = 1}^\infty B_{\rho_i} (x_i).
\end{cases}
$$
By the regularity of $\sigma$- harmonic functions (see, e.g., \cite{Si-07}), we know that $|\nabla h(x)| < C$ for $|x| \leq 1$.  Hence, for $x \in B_{\rho_j} (x_j)$,
$$
\aligned
|\nabla u_0 (x)| \leq C \int_{B_4} \frac{H( y, u_0 (y) )}{ |x - y|^{n - 2 \sigma + 1} } dy + C 
\leq C [I_1 (x) + I_2 (x) + I_3 (x)] + C, 
\endaligned
$$
where
$$
I_1 (x) : = \int_{B_{\rho_j} (x_j)} \frac{M_j}{ |x - y|^{n - 2 \sigma +1} } dy \leq C M_j \rho_j^{2 \sigma - 1} \leq C M_j^\frac{1}{2 \sigma} ~~~~~~ \textmd{for} ~ x \in B_{\rho_j} (x_j)
$$
by Lemma \ref{sim}, and
$$
\aligned
I_2 (x) & : = \sum\limits_{i = 1, i \neq j}^\infty \int_{B_{\rho_i} (x_i)} \frac{M_i}{ |x - y|^{n - 2 \sigma + 1} } dy \leq C \sum\limits_{i = 1, i \neq j}^\infty \frac{M_i \rho_i^n}{ \textmd{dist} ( B_{\rho_i} (x_i), B_{\rho_j} (x_j) )^{n - 2 \sigma + 1} } \\
& \leq C \sum\limits_{i = 1, i \neq j}^\infty \frac{ \rho_i^{n - 2 \sigma} }{ 2^i (\rho_i +\rho_j)^{n - 2 \sigma + 1} } \leq \frac{C}{\rho_j} \sim C 2^\frac{j}{2 \sigma} M_j^\frac{1}{2 \sigma} \leq C M_j^\frac{1+\beta}{2 \sigma} ~~~~~~ \textmd{for} ~ x \in B_{\rho_j} (x_j)
\endaligned
$$
by Lemma \ref{sim}, \eqref{distB} and \eqref{3seq2}, and 
$$
I_3 (x) : = \int_{ B_4 - \bigcup\limits_{i = 1}^\infty B_{\rho_i} (x_i) } \frac{1}{ |x - y|^{n - 2 \sigma + 1} } dy \leq C ~~~~~~ \textmd{for} ~ x \in B_{\rho_j} (x_j).
$$
Thus
\begin{equation}\label{nablau0Mj}
|\nabla u_0| \leq C M_j^\frac{1 + \beta}{2 \sigma} ~~~~~~ \textmd{in} ~ B_{\rho_j} (x_j), ~ j \geq 1.
\end{equation}
Since $n > 2 \sigma + 3$, it follows from \eqref{nablau0Mj} that
\begin{equation}\label{estu0Mj}
\frac{|\nabla u_0|}{ M_j^\frac{(1 - \beta) (n - 2 \sigma)}{4 \sigma} } \leq \frac{ C M_j^\frac{1 + \beta}{2 \sigma} }{ M_j^\frac{3 (1 - \beta)}{4 \sigma} } \leq \frac{C}{ M_j^\frac{1 - 5 \beta}{4 \sigma} }.
\end{equation}

Finally, we estimate the second term in \eqref{balabala}. Let
$$
s_j : = \inf \{ s > 0 ~ : ~  S \cap B_{\rho_j} (x_j) \subset B_s (x_j) \}
$$
and $\widetilde{u}_j (s) : = \psi (s, \lambda_j)$. Then $s_j \leq \rho_j$ and $\widetilde{u}_j (s) = u_j (x)$ when $|x - x_j| = s$. By \eqref{ujMjj1} we have
$$
\widetilde{u}_j (s) \geq C M_j^\frac{(1 - \beta) (n - 2 \sigma)}{4 \sigma} ~~~~~~ \textmd{for} ~ 0 \leq s \leq s_j, ~ j \geq 1.   
$$
It follows from Lemma \ref{sim} that
$$
\bigg( \frac{\lambda_j}{\lambda_j^2 + s_j^2} \bigg)^{2 \sigma} \geq C M_j^{1 - \beta} \geq C \bigg( \frac{ \varepsilon_j^\frac{2 \sigma}{n - 2 \sigma} }{2^j \lambda_j^\sigma} \bigg)^{1 - \beta}
$$
and hence, by \eqref{3seq2}, we have for $j \geq 1$,
$$
\aligned
s_j & \leq C \Bigg( \frac{2^j}{ \varepsilon_j^\frac{2 \sigma}{n - 2 \sigma} } \Bigg)^\frac{1 - \beta}{4 \sigma} \lambda_j^\frac{3 - \beta}{4} \leq C \big( \lambda_j^{- \beta} \big)^\frac{1 - \beta}{4 \sigma} \lambda_j^\frac{3 - \beta}{4} \leq C \big( \lambda_j^{- \beta} \big)^\frac{1}{4 \sigma} \lambda_j^\frac{3 - \beta}{4} \leq C \lambda_j^\frac{3 \sigma - \beta (\sigma + 1)}{4 \sigma}.
\endaligned
$$
Since $n > 2 \sigma + 3$, we have for $0 \leq s \leq s_j$ and $j \geq 1$ that
\begin{equation}\label{esttildeuj1}
\frac{- \widetilde{u}_j '(s)}{\widetilde{u}_j (s)^2} \leq C \lambda_j^\frac{\sigma (n - 2 \sigma - 3) - \beta (\sigma + 1)(n - 2 \sigma - 1)}{4 \sigma}.
\end{equation}

Since $n > 2 \sigma + 3$, we can take 
$$
\beta = \frac{(n - 2 \sigma - 3) \sigma}{3 (\sigma + 1)(n - 2 \sigma -1)} >0. 
$$
By \eqref{balabala}-\eqref{esttildeuj1} we get 
\begin{equation}\label{nablaKMl}
|\nabla K| \leq C \Bigg( \frac{1}{ M_j^\frac{3 (n - 2 \sigma - 1) - 2 \sigma (n - 2 \sigma - 6)}{12 \sigma (\sigma + 1) (n - 2 \sigma - 1)} } + \lambda_j^\frac{n - 2 \sigma - 3}{6} \Bigg) ~~~~~~ \textmd{in} ~ S \cap B_{\rho_j} (x_j), ~ j \geq 1.
\end{equation}
Hence, it follows from Lemma \ref{sim}, \eqref{3seq2}, \eqref{SsubBi} and \eqref{SB2SB} that \eqref{limnablaK} also holds for $x \in S$. Thus we have  $K \in C^1 (\R^n)$.

By sufficiently increasing $k_i$ for each $i \geq 1$, we can force $\kappa$ to satisfy
\begin{equation}\label{kkappaC1}
\| k - \kappa \|_{C^1 (\R^n)} < \frac{\varepsilon}{4}
\end{equation}
by \eqref{1kir}, \eqref{defka}, \eqref{blaka}, \eqref{KxH} and \eqref{keqKeqk}. We also can force $K$ to satisfy
$$
|\nabla (K - \kappa)| = |\nabla ( K - (\kappa - k) )| \leq |\nabla K| + \frac{\varepsilon}{4} \leq \frac{\varepsilon}{2} ~~~~~~ \textmd{in} ~ S
$$
by \eqref{SsubBi}, \eqref{SB2SB} and \eqref{nablaKMl}. Thus by \eqref{nkaeqnK}, $|\nabla (K - \kappa)| \leq \varepsilon/2$ in $\R^n$. It follows from \eqref{kKk} and \eqref{kkappaC1} that $K$ satisfies \eqref{K-k}.

\vskip0.1in

{\it Step 7. Extend the solution $u$ to the half space.}  By Step 5, the positive function $u$ in \eqref{Solu=0h87} belongs to $L_\sigma (\R^n) \cap C^{2 \sigma + \gamma} (\R^n \setminus \{ 0 \})$ and satisfies  
$$
(- \Delta)^\sigma u = K(x) u^\frac{n + 2 \sigma}{n - 2 \sigma} ~~~~~~ \textmd{in} ~ \R^n \setminus \{ 0 \}, 
$$  
where $K(x)$ is defined in \eqref{defK} which is in $C^1(\mathbb{R}^n)$ according to Step 6.  Moreover, it follows from \eqref{uixi}, \eqref{sinu1} and \eqref{defu0} that $u$ satisfies \eqref{u-phi} and \eqref{u-inf}. 

Set 
$$
\aligned
U (y, t) := \mathcal{P}_\sigma [u] (y, t) & = \gamma_{n, \sigma} \int_{\R^n} \frac{ t^{2 \sigma}  }{ (|y - x|^2 + t^2)^{(n + 2 \sigma)/2} } u(x) dx \\
& = \gamma_{n, \sigma} \int_{\R^n} \left( \frac{1}{1 + |z|^2} \right)^\frac{n + 2 \sigma}{2} u(y - t z) dz. 
\endaligned
$$
 Then $U(y, t)$ is well-defined for any $(y, t) \in \overline{\R_+^{n + 1}} \setminus \{ 0 \}$ and $U \in C^2 (\R_+^{n + 1})$. For any $R > r > 0$, let $\tau$ be a $C^\infty$ cut-off function on $\R^n$ such that ${\rm supp} ~ \tau \subset B_{r/2}$ and $\tau = 1$ in $B_{r/4}$. Rewrite $U = \mathcal{P}_\sigma [(1 - \tau) u] + \mathcal{P}_\sigma [\tau u]$. Since both $\mathcal{P}_\sigma [(1 - \tau) u]$ and $\mathcal{P}_\sigma [\tau u]$ belong to $W^{1, 2} (t^{1 - 2 \sigma}, \mathcal{B}_R^+ \setminus \overline{\mathcal{B}_r^+})$,  we have $U \in W^{1, 2} (t^{1 - 2 \sigma}, \mathcal{B}_R^+ \setminus \overline{\mathcal{B}_r^+})$.  By the extension formulation in \cite{CS-07},  we have for any $\Psi \in C_c^\infty (\overline{\R_+^{n + 1}} \setminus \{ 0 \})$,
$$
\aligned
\int_{\R_+^{n + 1}} t^{1 - 2 \sigma} \nabla U \nabla \Psi & = \int_{{\rm supp}\, \Psi} t^{1 - 2 \sigma} \nabla U \nabla \Psi = \int_{{\rm supp}\, \Psi \cap \R^n} \Psi \frac{\partial U}{\partial \nu^\sigma} \\
& = \int_{{\rm supp}\, \Psi \cap \R^n} \Psi K U^\frac{n + 2 \sigma}{n - 2 \sigma} = \int_{\R^n} \Psi K U^\frac{n + 2 \sigma}{n - 2 \sigma}.
\endaligned
$$
It follows that $U$ is a positive weak solution of \eqref{UKRN}. The proof of Theorem \ref{Lar}  is completed.  
\hfill$\square$

\vskip0.40in

\noindent {X. Du and H. Yang}\\
Department of Mathematics, The Hong Kong University of Science and Technology \\
Clear Water Bay, Kowloon, Hong Kong, China \\
E-mail addresses: xduah@connect.ust.hk (X. Du) ~~~~~~ mahuiyang@ust.hk (H. Yang)


\begin{thebibliography}{10}
\bibitem{BN} H. Berestycki, L. Nirenberg, On the method of moving planes and the sliding method. {\it Bol. Soc. Brasil. Mat. (N.S.)} {\bf  22} (1991), no. 1, 1-37. 


\bibitem{Cabre}   X. Cabr\'e, Y. Sire,  Nonlinear equations for fractional Laplacians, I: Regularity, maximum principles, and Hamiltonian estimates. {\it 
Ann. Inst. H. Poincar\'e Anal. Non Lin\'eaire} {\bf 31}  (2014), no. 1,  23-53. 

\bibitem{CGS} L. Caffarelli, B. Gidas, J. Spruck, Asymptotic symmetry and local behavior of semilinear elliptic equations with critical Sobolev growth. {\it Comm. Pure Appl. Math.}  {\bf 42}  (1989), 271-297. 


\bibitem{CJSX} L. Caffarelli, T. Jin, Y. Sire, J. Xiong, Local analysis of solutions of fractional semi-linear elliptic equations with isolated singularities. {\it Arch. Ration. Mech. Anal.} {\bf 213} (2014), no. 1, 245-268.


\bibitem{CS-07} L. Caffarelli, L. Silvestre, An extension problem related to the fractional Laplacian. {\it Comm. Partial Differential Equations} {\bf 32} (2007), no. 7-9, 1245-1260.


\bibitem{CG-11} S.-Y. A. Chang, M. Gonz\'alez, Fractional Laplacian in conformal geometry. {\it Adv. Math.} {\bf 226} (2011), no. 2, 1410-1432.


\bibitem{CXY} S.-Y. A. Chang, X. Xu, P. Yang, A perturbation result for prescribing mean curvature. {\it Math. Ann.} {\bf 310} (1998), no. 3, 473-496.


\bibitem{CL-97} C.-C. Chen, C.-S. Lin, Estimates of the conformal scalar curvature equation via the method of moving planes. {\it Commun. Pure Appl. Math.} {\bf 50} (1997), 971-1017.


\bibitem{DMA} Z. Djadli, A. Malchiodi, M. Ahmedou, The prescribed boundary mean curvature problem on $\mathbb{B}^4$. {\it J. Differential Equations} {\bf 206} (2004), no. 2, 373-398. 


\bibitem{DYla} X. Du, H. Yang, Large singular solutions for conformal $Q$-curvature equations on $\Sp^n$. {\it J. Differential Equations} {\bf 280} (2021), 618-643.


\bibitem{Es-96} J. F. Escobar, Conformal metrics with prescribed mean curvature on the boundary. {\it Calc. Var. Partial Differential Equations} {\bf 4} (1996), 559-592.


\bibitem{EG} J. F. Escobar, G. Garc\'ia, Conformal metrics on the ball with zero scalar curvature and prescribed mean curvature on the boundary. {\it J. Funct. Anal.} {\bf 211} (2004), 71-152.


\bibitem{GMS} M. Gonz\'alez, R. Mazzeo, Y. Sire, Singular solutions of fractional order conformal Laplacians. {\it J. Geom. Anal.} {\bf  22} (2012), no. 3, 845-863.


\bibitem{GQ}M. Gonz\'alez, J.  Qing, Fractional conformal Laplacians and fractional Yamabe problems. {\it  Anal. PDE}  {\bf 6} (2013), no. 7, 1535-1576. 


\bibitem{JLX-14} T. Jin, Y.Y. Li, J. Xiong, On a fractional Nirenberg problem, part I: blow up analysis and compactness of solutions. {\it J. Eur. Math. Soc. (JEMS)} {\bf 16} (2014), no. 6, 1111-1171.


\bibitem{JLX-15} T. Jin, Y.Y. Li, J. Xiong, On a fractional Nirenberg problem, part II: Existence of solutions. {\it Int. Math. Res. Not. IMRN} 2015, no. 6, 1555-1589.  


\bibitem{JDSX} T. Jin, O. de Queiroz, Y. Sire, J. Xiong, On local behavior of singular positive solutions to nonlocal elliptic equations. {\it Calc. Var. Partial Differential Equations}  {\bf 56} (2017), no. 1, Art. No. 9, 25 pp.

\bibitem{JY}  T. Jin, H. Yang, Local estimates for conformal $Q$-curvature equations. arXiv:2107.04437.  


\bibitem{KMPS} N. Korevaar, R. Mazzeo, F. Pacard, R. Schoen, Refined asymptotics for constant scalar curvature metrics with isolated singularities.
{\it Invent. Math.} {\bf 135} (1999), no. 2, 233-272.


\bibitem{Leung-03} M. C. Leung, Blow-up solutions of nonlinear elliptic equations in $\R^n$ with critical exponent. {\it Math. Ann.} {\bf 327} (2003), no. 4, 723-744.


\bibitem{LZ-03} Y.Y. Li, L. Zhang, Liouville-type theorems and Harnack-type inequalities for semilinear elliptic equations. {\it J. Anal. Math.} {\bf 90} (2003), 27-87.


\bibitem{LZ-95} Y.Y. Li, M. Zhu, Uniqueness theorems through the method of moving spheres. {\it Duke Math. J.} {\bf 80} (1995), 383-418.


\bibitem{Lin-00} C.-S. Lin, Estimates of the scalar curvature equation via the method of moving planes III, {\it Commun. Pure Appl. Math.} {\bf 53} (2000), 611-646.


\bibitem{Ro-14} X. Ros-Oton, J. Serra, The Dirichlet problem for the fractional Laplacian: regularity up to the boundary. {\it J. Math. Pures Appl. (9)} {\bf 101} (2014), no. 3, 275-302.  


\bibitem{Si-07} L. Silvestre, Regularity of the obstacle problem for a fractional power of the Laplace operator. {\it Comm. Pure Appl. Math.} {\bf 60} (2007), 67-112.  


\bibitem{Ta-05} S. Taliaferro, Existence of large singular solutions of conformal scalar curvature equations in $\mathbb{S}^n$.  {\it J. Funct. Anal.}  {\bf 224} (2005), no. 1, 192-216. 


\bibitem{TZ-06} S. Taliaferro, L. Zhang, Asymptotic symmetries for conformal scalar curvature equations with singularity. {\it Calc. Var. Partial Differential Equations} {\bf 26} (2006), 401-428. 


\bibitem{Zhang-02} L. Zhang, Refined asymptotic estimates for conformal scalar curvature equation via moving sphere method. {\it J. Funct. Anal.} {\bf 192} (2002), no. 2, 491-516. 


\end{thebibliography}
\end{document}